\documentclass[11pt]{article}

\usepackage[T1]{fontenc}
\usepackage[utf8]{inputenc}
\usepackage{authblk}

\usepackage{amssymb,amsmath,amsthm}
\usepackage[colorlinks=true,linktoc=page,urlcolor=blue,
citecolor=red,linkcolor=blue,pdfpagelabels,bookmarksnumbered,bookmarksopen,pagebackref=true]{hyperref}
\usepackage{graphicx}
\usepackage[english]{babel}
\usepackage{epsfig}
\usepackage{graphics}

\usepackage{amsthm}
\theoremstyle{plain}

\newtheorem{lemma}{Lemma}[section]

\newtheorem{corollary}[lemma]{Corollary}
\newtheorem{rem}[lemma]{Remark}

\theoremstyle{definition}

\theoremstyle{plain}
\newtheorem*{conj*}{Conjecture}
\newtheorem*{corollary*}{Corollary}

\newtheorem{prop}[lemma]{Proposition}
\makeatletter
\renewcommand*\l@section[2]{%
  \ifnum \c@tocdepth >\z@
    \setlength\@tempdima{1.5em}%
    \begingroup
      \parindent \z@ \rightskip \@tocrmarg
      \parfillskip -\rightskip
      \leavevmode
      \hangindent\@tempdima \hangafter\@ne
      #1\nobreak\hfil \nobreak\hb@xt@\@pnumwidth{\hss #2}\par
    \endgroup
  \fi}
\makeatother
\theoremstyle{plain}
\newtheorem{thmletter}{Theorem}

\def\div{\mathop{\mathrm{div}}}

\newcommand{\R}{{\mathbb R}}

\newcommand{\DR}{{\cal D}}

\newcommand{\mn}{\medbreak\noindent}

\def\bean#1\eean{\begin{eqnarray*}#1\end{eqnarray*}}

\begin{document}

\title{\bfseries\Large Unique continuation, nonexistence and bubbling\\ for the critical $p$-Laplace equation in the plane}
\author{Carlo Mercuri}
\affil{

Dipartimento di Scienze Fisiche, Informatiche e Matematiche\\\small 
Universit\`a di Modena e Reggio Emilia\\\small
Via Campi 213/b, 41125 Modena, Italy\\\small
\textit{carlo.mercuri@unimore.it}
}
\date{}

\maketitle
{\renewcommand{\thefootnote}{}}

\begin{abstract}
For $1<p<2$ we prove weak and strong unique continuation properties for the solutions to $\Delta_pu+f(u)=0$
in a planar domain, with $f$ continuous and such that $|f(s)|\leq C|s|^{p-1}$: a solution vanishing on
an open set vanishes identically, and so does a solution vanishing to infinite order
at a single point. We can therefore make progress towards a proof of some
long-standing nonexistence results available only for $p=2$, such as that of Esteban and Lions~\cite{EstebanLions}, establishing here that for the critical $p$-Laplace
equation with zero Dirichlet boundary condition on a half-plane, there are no nontrivial finite
energy solutions. In the plane, unique continuation thus provides the missing ingredient for a
generalisation to the $p$-Laplacian operator of a classical result of
Struwe~\cite{Struwe} on the bubble-profile decomposition of possibly sign-changing Palais--Smale sequences associated to the
Brezis--Nirenberg problem~\cite{BrezisNirenberg}, which remains open in dimension $N\geq3$ for
$p\neq2$. We obtain a characterisation of their loss of compactness in terms of the finite
energy solutions of $\Delta_pu+|u|^{p^*-2}u=0$ in $\R^2$.
\end{abstract}

\setcounter{tocdepth}{1}
\tableofcontents

\section{Introduction}
\label{sec:construction}

Let $\Omega\subset\R^2$ be open, $1<p<2$ and $f:\R\to\R$ continuous. A weak solution of
\begin{equation}
\label{eq:eq-f}
-\Delta_pu=f(u)\qquad\mbox{in }\Omega ,\qquad \Delta_pu=\div\big(|\nabla u|^{p-2}\nabla u\big),
\end{equation}
is a $u\in W^{1,p}_{{\textrm{loc}}}(\Omega)$ with
\begin{equation*}
\int_\Omega|\nabla u|^{p-2}\nabla u\cdot\nabla\varphi\,dx=\int_\Omega f(u)\varphi\,dx
\end{equation*}
for every $\varphi\in C_c^\infty(\Omega)$.

\noindent In this paper we prove weak and strong unique continuation properties for the solutions of \eqref{eq:eq-f}, allowing a zeroth-order term $f(u)$ not present in the classical results, available in fact only for $p$-harmonic functions; see Iwaniec and Manfredi~\cite{IwaniecManfredi} and Manfredi~\cite{Manfredi}. Our approach is in their spirit. In fact, setting $h=\partial_zu=\frac12(u_x-iu_y)$ for the complex gradient, we have $|\nabla u|=2|h|$, and the equation for $u$
becomes a first-order equation for $h$, of Beltrami type,
\begin{equation*}
\partial_{\bar z}h=\mu \partial_zh+\nu\overline{\partial_z h}+\alpha h+\beta u ,
\end{equation*}
with $\mu$ and $\nu$ measurable and
$|\mu|+|\nu|\leq\frac{2-p}p<1$, namely a uniform ellipticity property, and $\alpha$, $\beta$
bounded and measurable.

For $f=0$ the last two terms are absent, $\partial_{\bar z}h=\mu \partial_zh+\nu\overline{\partial_z h}$, and $h$ is
quasiregular: $|\partial_{\bar z}h|\leq\frac{2-p}p\,|\partial_zh|$ almost everywhere (Bojarski and
Iwaniec~\cite{BojarskiIwaniec} for $p\geq2$; Manfredi~\cite[Theorem~1]{Manfredi} and Iwaniec and
Manfredi~\cite[(12)--(13), p.~4]{IwaniecManfredi} for every $p>1$). Nonconstant quasiregular mappings are
discrete, and this is how Manfredi~\cite[Corollary 1]{Manfredi} remarkably obtained the
isolatedness of the critical points of a nonconstant $p$-harmonic function.

Unlike in the $p$-harmonic case, the source term $f(u)$ produces the two extra terms above. The
term $\beta u$ is the one which will involve most of the work, because in principle it is
independent of the unknown $h$, so that the resulting equation is not homogeneous. Even so, near
any point where $u$ vanishes without vanishing identically, we construct a suitable perturbation
of $h$ by a multiple of $u$ whose zeros are isolated, with arguments inspired by
Alessandrini~\cite{Alessandrini2012}.

To show that the weak unique continuation property holds, we can argue as follows. Let
\begin{equation*}
\Omega_0=\{z\in\Omega:\ u\equiv0\ \mbox{on some disk around }z\}
\end{equation*}
be the interior of the nodal set of $u$. By definition it is an open set, and it is nonempty as
soon as $u$ vanishes on some open subset of $\Omega$; if it is also closed in $\Omega$, by
connectedness we can finally obtain $\Omega_0=\Omega$, namely $u\equiv0$, as we wanted. The issue then boils down to proving $\Omega_0$ being a closed set, and we deliver this by observing that at an
accumulation point $z_0\in\Omega\setminus\Omega_0$ of $\Omega_0$, $u$ vanishes but not identically
on any disk around $z_0$. Let $(z_n)_{n\in\mathbb N}\subset\Omega_0$ with $z_n\to z_0$, and let
$B_R(z_0)$ be a disk around $z_0$ on which the perturbation of $h$ just mentioned is defined, its
zeros in $B_R(z_0)$ being isolated, as we shall see later. For $n$ large $z_n\in B_R(z_0)$, and $u\equiv0$ on a disk
$D\subset B_R(z_0)$ centred at $z_n$; there $h$ vanishes as well, and so does the
perturbation. Its zeros in $B_R(z_0)$ are then not isolated, a contradiction, so no such $z_0$
exists.

The aforementioned perturbation of $h$ via $u$ takes the form
\begin{equation*}
\mathcal R=h-vu ,
\end{equation*}
where $v$ is chosen so that the term in $u$ cancels. The condition on $v$ turns out to be a
Riccati-type equation, which can be solved on the disk $B_R(z_0)$ above provided $R$ is small
enough, and with such a $v$ the function $\mathcal R$ solves the homogeneous system
\begin{equation*}
\partial_{\bar z}\mathcal R=\mu \partial_z\mathcal R+\nu\overline{\partial_z \mathcal R}+A\mathcal R+B\overline{\mathcal R}
\end{equation*}
with $A$ and $B$ bounded. Moreover $\mathcal R$ does not vanish identically near $z_0$: otherwise
$\nabla u$ would be controlled by $u$, and by a Gronwall-type argument $u$ would vanish there.

A reason for reducing to such a system is because we can then take advantage of the representation theorem of Bers and
Nirenberg~\cite[p.~116]{BersNirenberg} (see also
Bojarski~\cite[Theorem 4.4, p.~477]{Bojarski}): $\mathcal R=\mathcal F\circ\chi$, with $\chi$ a
quasiconformal homeomorphism of the plane and $\mathcal F$ a holomorphic function times a
nonvanishing continuous factor. The zeros of $\mathcal R$ are then those of a holomorphic function,
carried by $\chi$, and therefore isolated --- which is what we used in the argument above. This
representation replaces Stoilow's factorization, used by Iwaniec and
Manfredi~\cite{IwaniecManfredi}, and the discreteness of quasiregular mappings, used by
Manfredi~\cite{Manfredi}.

For the strong unique continuation property, the same representation gives more useful information. In fact, near a zero, $\mathcal R$ vanishes like a
power of $|\chi(z)-\chi(z_0)|$, and $\chi$, being quasiconformal, does not contract distances
faster than a power; so $\mathcal R$ cannot vanish faster than a power of $|z-z_0|$. A Caccioppoli-type
inequality carries that bound from $\mathcal R$ over to $u$, which is the property of vanishing at finite order
of Theorem~\ref{thm:ucp}(ii).
\mn
With Theorem~\ref{thm:ucp} at hand, we can prove, for critical $p$-Laplace equations in the plane,
nonexistence results so far available only for $p=2$, such as that of Esteban and
Lions~\cite{EstebanLions}: for the critical $p$-Laplace equation with zero Dirichlet boundary
condition on a half-plane, there are no nontrivial finite energy solutions. This result has been proved for nonnegative solutions on the half-space $\R^N_+$ and all
$1<p<N$ in \cite[Theorem 1.1, p.~470]{MercuriWillem}, via the strong maximum principle of
V\'azquez~\cite[Theorem 5, p.~200]{Vazquez}. 
In the third section we also deal with tools, useful for the compactness analysis of Brezis--Nirenberg type problems~\cite{BrezisNirenberg} involving the $p$-Laplacian operator and critical nonlinearities. Loosely speaking, the aforementioned nonexistence result allows to obtain, for critical $p$-Laplace problems in the plane, a generalisation of Struwe's~\cite{Struwe} global compactness result for possibly sign-changing Palais--Smale sequences. Namely, a characterisation of their loss of
compactness given in terms of the bubble-type profile solutions to $\Delta_pu+|u|^{p^*-2}u=0$ in $\R^2.$ This result, in the case of $\R^2$ and $1<p<2,$ allows one to drop a hypothesis on the negative parts of the PS sequences used in \cite[Theorem 1.2, p.~471]{MercuriWillem} and some of its recent variants.

\subsection{Unique continuation for the $p$-Laplace equation with a zeroth-order term}

\begin{thmletter}[Unique continuation]
\label{thm:ucp}
Let $1<p<2$, let $\Omega\subset\R^2$ be open and connected, let $f:\R\to\R$ be
continuous with
\begin{equation}
\label{eq:f-growth}
|f(s)|\leq C_M|s|^{p-1}\qquad\mbox{for }|s|\leq M ,
\end{equation}
for every $M>0$, and let $u\in C^{1,\alpha}_{{\textrm{loc}}}(\Omega)$ be a weak solution of
\eqref{eq:eq-f}.
\begin{enumerate}
\item[(i)] If $u$ vanishes on a nonempty open subset of $\Omega$, then $u\equiv0$ in $\Omega$.
\item[(ii)] If $u\not\equiv0$, then for every $z_0\in\Omega$ there are $m\geq0$ and $c,\rho_0>0$ with
\begin{equation}
\label{eq:finite-order}
\int_{B_\rho(z_0)}|u|^p\,dx\ \geq\ c\,\rho^{m}\qquad\mbox{for }0<\rho<\rho_0 ,
\end{equation}
so that $u$ vanishes to a finite order at each of its zeros. In particular, a solution which vanishes to infinite order at one point, in the sense that the integral in \eqref{eq:finite-order} is
$O(\rho^j)$ as $\rho\to0$ for every $j\in\mathbb N$, vanishes identically.
\end{enumerate}
\end{thmletter}

The growth assumption \eqref{eq:f-growth} is sharp, in the sense that counterexamples can be
constructed for nonlinearities with power growth $q<p-1$. Indeed, for the equation
$-\Delta_pu+\beta(u)=0$, with $\beta$ continuous, nondecreasing and $\beta(0)=0$, the strong
maximum principle for nonnegative solutions holds if and only if
$\int_{0^+}(\beta(s)s)^{-1/p}\,ds=\infty$, by V\'azquez~\cite{Vazquez}. For $\beta(u)=|u|^{q-1}u$ this integral is finite precisely when
$q<p-1$, and there are dead-core solutions, namely nontrivial nonnegative solutions vanishing on
an open set.

\begin{corollary*}
Let $1<p<2$, let $\Omega\subset\R^2$ be open and connected, let $\lambda\in\R$ and let $u$ be a weak
solution of $-\Delta_pu=\lambda|u|^{p-2}u$ in $\Omega$. If $u$ vanishes on a nonempty open subset of
$\Omega$, then $u\equiv0$; otherwise $u$ vanishes to finite order at each of its zeros.
\end{corollary*}

\noindent Indeed \eqref{eq:f-growth} holds with equality here, and such solutions are locally
bounded, hence $C^{1,\alpha}_{{\textrm{loc}}}$.

\begin{rem}
\label{rem:gradient}
The proof of Theorem~\ref{thm:ucp} is robust enough to accommodate more general equations, for
instance
\begin{equation*}
-\Delta_pu=f(x,u,\nabla u)\qquad\mbox{in }\Omega ,
\end{equation*}
with $f$ measurable in $x$ and continuous in $(u,\nabla u)$. The conclusions of
Theorem~\ref{thm:ucp} remain valid for weak solutions $u\in C^{1,\alpha}_{{\textrm{loc}}}(\Omega)$,
provided that for every $M>0$
\begin{equation*}
|f(x,s,\xi)|\leq C_M\big(|s|^{p-1}+|\xi|^{p-1}\big)\qquad\mbox{for a.e.\ }x\in\Omega,\ |s|\leq M,\ |\xi|\leq M .
\end{equation*}
We leave the analysis of these equations on planar domains out of this work.

\end{rem}

\subsubsection*{Related questions}
\begin{enumerate}
\item[1)] The proof of Theorem~\ref{thm:ucp} rests on a Beltrami-type system for the
complex gradient (see \eqref{eq:h-eq} below), that is, on the fact that $-\Delta_pu=f(u)$ is
a determined first-order system for $\nabla u$ in the plane. We do not know whether the
weak unique continuation property holds in dimension $N\geq3$, even for $f=0$. A promising
approach to higher dimensions would be to consider a symmetric version of
Theorem~\ref{thm:ucp}. In fact, for a solution invariant under a compact
group of isometries with two-dimensional orbit space, e.g.\ the axially symmetric solutions
$u=U(r,x_N)$, $r=|x'|$, the equation reduces to a planar $p$-Laplace equation with a
first-order term,
\begin{equation*}
-\Delta_pU=f(U)+\frac{N-2}{r}\,|\nabla U|^{p-2}\partial_rU
\end{equation*}
in the axial case, which is bounded by $C|\nabla U|^{p-1}$ away from the axis: this is an
equation of the type considered in Remark~\ref{rem:gradient}, to which the proof of
Theorem~\ref{thm:ucp} would apply.
\item[2)] \textcolor{red}{By Corollary~\ref{cor:isolated}, the points of $\Sigma_0$ are isolated,} a property which is in
fact stronger than that we actually use here. For linear equations in the plane with Lipschitz continuous principal coefficients
and bounded \textcolor{red}{first-order coefficients, no zeroth-order term,} \textcolor{red}{one has more:} every critical point of \textcolor{red}{a nonconstant solution} is isolated, $\partial_zu$ vanishes there with
finite multiplicity $m$, and the index of $\nabla u$ is $-m$; see Alessandrini and
Magnanini~\cite[p.~570]{AlessandriniMagnanini}. \textcolor{red}{Does an
analogous formula hold at a point of $\Sigma_0$ for $p\neq2$ in the instances discussed in
this paper?}
\item[3)] The equation of Theorem~\ref{thm:ucp} is the stationary case of
$u_t=\Delta_pu+f(u)$ on $\Omega\times(0,T)$, and for the parabolic flow the corresponding question is
open: if $u$ vanishes on a nonempty open subset $\omega\times(t_1,t_2)$, does $u(\cdot,t)$
vanish in $\Omega$ for every $t\in(t_1,t_2)$? For $p=2$ the answer is positive. There
$c=f(u)/u$ is locally bounded and $u$ solves $\Delta u-u_t+cu=0$; a solution of such an
equation which vanishes on $\omega\times(0,T)$ vanishes on the whole cylinder, and at each
time $u(\cdot,t_0)$ either vanishes identically or has no zero of infinite order, the
analogue of \eqref{eq:finite-order} (Vessella~\cite{Vessella}).

For $1<p<2$ one may try to argue at a fixed time, the slice being a solution of the planar
equation $-\Delta_pu=f(u)-u_t$ with $u_t$ read as a source. Our approach to
Theorem~\ref{thm:ucp} requires that source so satisfy a similar bound that enjoyed by $f(u)$,
namely $|u_t|\,|\nabla u|^{2-p}\leq C(|u|+|\nabla u|)$ locally and uniformly in $t$. The term
$f(u)$ obeys it, by the growth \eqref{eq:f-growth} and Young's inequality,
$|u|^{p-1}|\nabla u|^{2-p}\leq(p-1)|u|+(2-p)|\nabla u|$. For an argument similar to the proof
of Theorem~\ref{thm:ucp} to be carried out, this would be a too severe demand on $u_t$, being
forced to vanish wherever $u$ and $\nabla u$ do simultaneously. Indeed, this is false even in
very common situations: for $p=2$ the solution $t+\frac14|x|^2$ of the heat equation has
$u=0$, $\nabla u=0$ and $u_t=1$ at the origin. 
\end{enumerate}

\subsection{Nonexistence for critical equations}

Let us now set in \eqref{eq:eq-f} $f(s)=|s|^{p^*-2}s$, where $1<p<N$ and
$p^*=\frac{Np}{N-p}$ is the critical Sobolev exponent. We consider hereafter only finite energy solutions satisfying a Dirichlet boundary condition. On the half-space
$H=\{x\in\R^N:\ x_N>0\}$ we therefore consider solutions in $\DR_0^{1,p}(H)$, the closure of
$C_c^\infty(H)$ in $\DR^{1,p}(\R^N)=\{u\in L^{p^*}(\R^N):\ \nabla u\in L^p(\R^N)\}$: a weak
solution of
\begin{equation}
\label{eq:main}
-\Delta_pu=|u|^{p^*-2}u\qquad\mbox{in }H
\end{equation}
is a $u\in\DR_0^{1,p}(H)$ with
$\int_H|\nabla u|^{p-2}\nabla u\cdot\nabla\varphi=\int_H|u|^{p^*-2}u\,\varphi$ for every
$\varphi\in\DR_0^{1,p}(H)$. We first establish a positive answer for $N=2$ to a conjecture raised in \cite[p.~3]{FarinaMercuriWillem}, by the following

\begin{thmletter}[Nonexistence on the half-plane]
\label{thm:main}
Let $N=2$ and $1<p<2$. Then every weak solution
$u\in\DR_0^{1,p}(H)$ of \eqref{eq:main} vanishes identically.
\end{thmletter}

This theorem is proved for nonnegative solutions, for every $1<p<N$, in
\cite[Theorem 1.1, p.~470]{MercuriWillem}. For $p=2$ it holds for every solution, by a classical result of Esteban and Lions \cite[Theorem I.1, p.~2]{EstebanLions}, proved by means of a weak unique continuation principle, which for $1<p<2$, is provided by Theorem~\ref{thm:ucp}.

In fact, Theorem~\ref{thm:main} follows from Theorem~\ref{thm:ucp}(i) in three steps. The normal
derivative $\partial_Nu=\partial u/\partial x_N$ vanishes on $\partial H$; hence the extension
by zero, $v=u$ on $H$ and $v=0$ on $\R^N\setminus H$, is a weak solution of
$-\Delta_pv=|v|^{p^*-2}v$ in $\R^N$, vanishing on the open set $\R^N\setminus\overline H$; and
for $N=2$, Theorem~\ref{thm:ucp}(i) then gives $v\equiv0$.

For a bounded domain $\Omega\subset\R^N$, let us consider $W_0^{1,p}(\Omega),$ namely the closure of $C_c^\infty(\Omega)$ in
$W^{1,p}(\Omega).$ Weak solutions of $-\Delta_pu=|u|^{p^*-2}u$ in $\Omega$ are defined as for
\eqref{eq:main}, with $\Omega$ and $W_0^{1,p}(\Omega)$ in place of $H$ and $\DR_0^{1,p}(H)$.
We have the following nonexistence result.
\begin{thmletter}[Nonexistence in star-shaped domains]
\label{thm:star}
Let $1<p<2$ and let $\Omega\subset\R^2$ be a bounded domain with $C^2$ boundary, star-shaped with
respect to a point $x_0\in \Omega$:
\begin{equation}
\label{eq:star-shaped}
(x-x_0)\cdot\nu(x)\geq0\qquad\mbox{for every }x\in\partial \Omega ,
\end{equation}
$\nu$ the outer unit normal. Then every weak solution $u\in W_0^{1,p}(\Omega)$ of
\begin{equation}
\label{eq:eq-Omega}
-\Delta_pu=|u|^{p^*-2}u\qquad\mbox{in }\Omega
\end{equation}
vanishes identically.
\end{thmletter}

For $p=2$ this is Pohozaev's classical theorem \cite{Pohozaev}, as recalled in
Guedda and V\'eron~\cite[p.~879]{GueddaVeron}, proved for sign-changing solutions by the Pohozaev identity and
unique continuation (see e.g. Heinz \cite{Heinz}) in Struwe~\cite[Theorem III.1.3]{StruweBook}, which is the scheme
followed here; for the nonlinearity $|u|^{p^*-1}$, which forces
$u\geq0$, and every $1<p<N$ it is due to Guedda and V\'eron~\cite[Corollary 1.3, p.~886]{GueddaVeron}. For radial solutions the question has been solved for every $1<p<N$: a radial
$u\in W_0^{1,p}(B)$ solution to $-\Delta_pu=|u|^{p^*-2}u$ in a ball $B,$ vanishes identically by \cite[Proposition 1.2, p.~4]{FarinaMercuriWillem}, which in turn extended to the whole range
$1<p<N$ the radial nonexistence known for $\frac{2N}{N+2}\leq p\leq2$ from \cite[p.~482]{MercuriWillem}.

\subsection{Blow-up analysis of Palais--Smale sequences for critical boundary value problems}

Let now $\Omega\subset\R^2$ be a smooth bounded domain, 
$a\in L^{2/p}(\Omega),$ $\|u\|=\|\nabla u\|_{L^p}$ for all $u\in W_0^{1,p}(\Omega)$
and $u\in \DR^{1,p}(\R^2)$, and 
\begin{equation}
\label{eq:functionals}
\begin{split}
J_\Omega(u)&=\int_\Omega\frac{|\nabla u|^p}p+a\frac{|u|^p}p-\frac{|u|^{p^*}}{p^*}\,dx,
\qquad u\in W_0^{1,p}(\Omega),\\
J_{\R^2}(u)&=\int_{\R^2}\frac{|\nabla u|^p}p-\frac{|u|^{p^*}}{p^*}\,dx,
\qquad u\in\DR^{1,p}(\R^2).
\end{split}
\end{equation}
A Palais--Smale sequence for $J_\Omega$ is a sequence $(u_n)_{n\in\mathbb N}\subset W_0^{1,p}(\Omega)$ along which
$J_\Omega$ is bounded and $J_\Omega'$ tends to $0$ in $W^{-1,p'}(\Omega)$,
$p'=\frac p{p-1}.$ We say that a Palais-Smale sequence $(u_n)_{n\in\mathbb N}$ for $J_\Omega$ is at a level $c\in\R$ when $J_\Omega(u_n)\to c$.

\begin{thmletter}[Global compactness]
\label{thm:compact}
Let $1<p<2$, let $\Omega$, $a$ be as above and let $(u_n)_{n\in\mathbb N}$ be a Palais--Smale
sequence for $J_\Omega$ at a level $c$. Then, passing to a subsequence, there are an integer
$k\geq0$, a solution $v_0\in W_0^{1,p}(\Omega)$ of
\begin{equation*}
-\Delta_pu+a|u|^{p-2}u=|u|^{p^*-2}u\qquad\mbox{in }\Omega ,
\end{equation*}
nontrivial solutions $v_1,\dots,v_k\in\DR^{1,p}(\R^2)$ of
\begin{equation*}
-\Delta_pu=|u|^{p^*-2}u\qquad\mbox{in }\R^2 ,
\end{equation*}
and sequences $(\varepsilon_n^i)_{n\in\mathbb N}\subset(0,\infty)$, $(y_n^i)_{n\in\mathbb N}\subset\Omega$,
$i=1,\dots,k$, with
\begin{equation*}
\varepsilon_n^i\to0,\qquad
\frac{\operatorname{dist}(y_n^i,\partial\Omega)}{\varepsilon_n^i}\to\infty ,
\end{equation*}
such that, as $n\to\infty$,
\begin{equation}
\label{eq:decomposition}
\begin{split}
&\Big\|u_n-v_0-\sum_{i=1}^k(\varepsilon_n^i)^{\frac{p-2}p}\,v_i\Big(\frac{\cdot-y_n^i}{\varepsilon_n^i}\Big)\Big\|\to0,\\
&\|u_n\|^p\to\|v_0\|^p+\sum_{i=1}^k\|v_i\|^p,\\
&J_\Omega(v_0)+\sum_{i=1}^kJ_{\R^2}(v_i)=c .
\end{split}
\end{equation}
In particular, no limiting profile-solution to the critical half-plane equation occurs.
\end{thmletter}

For $p=2$ and $N\geq3$ this is Struwe's celebrated global compactness result \cite{Struwe}; see also Willem~\cite[Theorem 8.13]{Willem}. Pioneering blow-up analysis of Palais--Smale sequences
at points of concentration goes back to Sacks and Uhlenbeck
\cite{SacksUhlenbeck} and to Wente \cite{Wente}. Struwe's global compactness result extends the
concentration-compactness principle of P.-L. Lions~\cite{Lions1984,Lions1985} from minimization to minimax problems, the well-known link between the two being the celebrated variational principle of
Ekeland~\cite{Ekeland}: every minimizing sequence of a $C^1$ functional bounded below on a
Banach space lies at vanishing distance from a Palais--Smale sequence at the same level. For $1<p<N$ the above theorem is proved in \cite[Theorem 1.2, p.~471]{MercuriWillem} under the additional hypothesis $\|(u_n)^-\|_{L^{p^*}(\Omega)}\to0$, which is removed here in the plane, via Theorem \ref{thm:main}.

The summability assumption $a\in L^{2/p}(\Omega)$ is what makes the limiting
profiles of \eqref{eq:decomposition} solutions of the equation without the weight. Heuristically,
let $v\in C_c^\infty(\R^2)$ be supported in $B_R$ and consider the bubble
\begin{equation*}
v_n=\varepsilon_n^{\frac{p-2}p}\,v\Big(\frac{\,\cdot\,-y_n}{\varepsilon_n}\Big),
\qquad\varepsilon_n\to0 .
\end{equation*}
By Sobolev invariance, the gradient and the critical terms of $J_\Omega$ do not `see' the scaling, while the weight term
does:
\begin{equation}
\label{eq:weight-bubble}
\int_\Omega a\,|v_n|^p\,dx=\varepsilon_n^{\,p}\int a(y_n+\varepsilon_n y)\,|v(y)|^p\,dy .
\end{equation}
The factor $\varepsilon_n^{\,p}$ is exactly the one absorbed by the rescaling of the $L^{2/p}$
norm, so that H\"older's inequality gives
\begin{equation}
\label{eq:weight-holder}
\Big|\int_\Omega a\,|v_n|^p\,dx\Big|\ \leq\
\|v\|_{L^{p^*}(\R^2)}^p\ \|a\|_{L^{2/p}(B_{R\varepsilon_n}(y_n))}\ \longrightarrow\ 0 .
\end{equation}
Thus $L^{2/p}$ is the summability that makes the weight somehow invisible to the concentration. The threshold is
sharp: for $a=|x|^{-s}$, which lies in $L^{2/p}$ near the origin exactly when $s<p$, the
right-hand side of \eqref{eq:weight-bubble} is of order $\varepsilon_n^{\,p-s}$, so beyond $s=p$
one may expect limiting profiles of the type of those of the
Caffarelli--Kohn--Nirenberg equation \eqref{eq:CKN} below.

In fact, the challenge pops up if one weakens $a\in L^{2/p}(\Omega)$ to the scale-invariant condition that
\eqref{eq:weight-bubble} suggests,
\begin{equation*}
\lim_{r\to0}\ \sup_{y\in\overline\Omega}\ r^{\,p-2}\int_{B_r(y)\cap\Omega}a\,dx\ =\ 0 :
\end{equation*}
$\int_\Omega a\,|u|^p\,dx$ is not bounded for a general $u\in\DR^{1,p}(\R^2)$. One may then be led to consider hypotheses, for instance on some $a\geq 0$, in the form of
a weighted Poincar\'e inequality of Hardy type,
$\int_\Omega a\,|u|^p\,dx\leq C\int_\Omega|\nabla u|^p\,dx$; see the Fefferman--Phong
inequality~\cite{FeffermanPhong,Fefferman}.

When the domain has a symmetry with respect to a closed subgroup of $O(N)$, a similar
decomposition holds (\cite[Theorem 2.2]{MercuriPacella}): for $1<p<2=N,$ and
reasoning along the same lines of the above theorem, one can rule out that a limiting profile-solution to the critical half-plane equation occurs.
For $p=2$ it is the representation theorem of
Clapp~\cite{Clapp}, used in Clapp and Pacella~\cite[p.~587]{ClappPacella} to prove multiplicity results on symmetric domains with nontrivial topology. We will get back to the impact of symmetry in higher dimensional problem at the end of this Introduction.
\mn
In a functional setting which generalises that of Theorem~\ref{thm:compact} above without the
lower-order term,
Chernysh~\cite[Theorem 2, p.~3]{Chernysh} recently proved a representation theorem for arbitrary Palais--Smale sequences of
weighted critical $p$-Laplace equations, $1<p<N$. For $N=2$ and $1<p<2$,
Theorem~\ref{thm:main} rules out the limiting half-plane profile solutions, which in
Chernysh decomposition may arise only for equal weight exponents
(\cite[Theorem 1 and Theorem 2]{Chernysh}).

Let $1<p<2$, let $b<\frac{2-p}p$ and set
$\gamma=\frac{2-p(1+b)}p>0$. For an open $\mathcal O\subset\R^2$ let $\DR^{1,p}_b(\mathcal O,0)$ be the
completion of $C_c^\infty(\mathcal O)$ for the norm $\big(\int_{\mathcal O}|\nabla u|^p|x|^{-bp}dx\big)^{1/p}$,
and let
\begin{equation}
\label{eq:CKN}
-\div\big(|x|^{-bp}|\nabla u|^{p-2}\nabla u\big)=|x|^{-bp^*}|u|^{p^*-2}u
\qquad\mbox{in }\mathcal O ,\qquad u\in\DR^{1,p}_b(\mathcal O,0),
\end{equation}
with associated energy functional
\begin{equation*}
\widetilde J_{\mathcal O}(u)=\int_{\mathcal O}\frac{|\nabla u|^p}p\,|x|^{-bp}-\frac{|u|^{p^*}}{p^*}\,|x|^{-bp^*}\,dx.
\end{equation*}
 Equation \eqref{eq:CKN} is the Euler--Lagrange
equation for the optimisers for the Caffarelli, Kohn and
Nirenberg inequality~\cite{CaffarelliKohnNirenberg}, which here reads
$\big(\int_{\R^2}|u|^{p^*}|x|^{-bp^*}dx\big)^{1/p^*}\leq
C\big(\int_{\R^2}|\nabla u|^p|x|^{-bp}dx\big)^{1/p}$ for $u\in C^\infty_c(\R^2)$, and
makes $\widetilde J_{\mathcal O}$ well defined on $\DR^{1,p}_b(\mathcal O,0).$ On $\Omega$ the limiting profiles are of two kinds: solutions of
\eqref{eq:CKN} on $\mathcal O=\R^2$, which concentrate at the origin, and solutions of the unweighted
equation $-\Delta_pu=|u|^{p^*-2}u$, on $\R^2$ or on a half-plane, with energy $J_{\R^2}$ of
\eqref{eq:functionals}. The unweighted profiles are `brought' into the weighted
space by a transform ~\cite[p.~3]{Chernysh}: for $y\in\R^2$, $\lambda>0$ and
$\varphi\in\DR^{1,p}(\R^2)$,
\begin{equation*}
\begin{split}
&\tau_{y,\lambda}\varphi=\lambda^{-\gamma}\varphi\Big(\frac{\cdot-y}\lambda\Big)
\qquad\mbox{for }b=0,\\
&\tau_{y,\lambda}\varphi=\Big(\frac\lambda{|y|}\Big)^{-b}\lambda^{-\gamma}
\varphi\Big(\frac{\cdot-y}\lambda\Big)\,\eta\Big(\frac{2(\cdot-y)}{|y|}\Big)
\qquad\mbox{for }b\neq0,
\end{split}
\end{equation*}
$\eta\in C_c^\infty(\R^2;[0,1])$ supported in $B_1(0)$ and equal to $1$ on $B_{1/2}(0)$; for
$b=0$, $\gamma=\frac{2-p}p$ and $\tau_{y,\lambda}$ is the scaling of \eqref{eq:decomposition}.

\begin{thmletter}[Weighted equations in the plane]
\label{thm:chernysh}
Let $1<p<2$, let $\Omega\subset\R^2$ be a bounded $C^1$ domain containing the origin, let
$b<\frac{2-p}p$ and let $(u_n)_{n\in\mathbb N}$ be a Palais--Smale sequence for \eqref{eq:CKN} on
$\mathcal O=\Omega$ at a level $c$. Then, passing to a subsequence, there are integers $k,l\geq0$, a
solution $v_0\in\DR^{1,p}_b(\Omega,0)$ of \eqref{eq:CKN} on $\Omega$, nontrivial solutions
$w_1,\dots,w_k\in\DR^{1,p}_b(\R^2,0)$ of \eqref{eq:CKN} on $\R^2$, nontrivial solutions
$v_1,\dots,v_l\in\DR^{1,p}(\R^2)$ of
\begin{equation*}
-\Delta_pu=|u|^{p^*-2}u\qquad\mbox{in }\R^2 ,
\end{equation*}
and sequences $(\varrho_n^j)_{n\in\mathbb N}\subset(0,\infty)$, $j=1,\dots,k$, and
$(\varepsilon_n^i)_{n\in\mathbb N}\subset(0,\infty)$, $(y_n^i)_{n\in\mathbb N}\subset\Omega$, $i=1,\dots,l$, with
\begin{equation*}
\varrho_n^j\to0,\qquad \varepsilon_n^i\to0,\qquad
\frac{|y_n^i|}{\varepsilon_n^i}\to\infty,\qquad
\frac{\operatorname{dist}(y_n^i,\partial\Omega)}{\varepsilon_n^i}\to\infty ,
\end{equation*}
such that, as $n\to\infty$,
\begin{equation*}
\begin{split}
&\Big\|u_n-v_0-\sum_{j=1}^k(\varrho_n^j)^{-\gamma}w_j\Big(\frac\cdot{\varrho_n^j}\Big)
-\sum_{i=1}^l\tau_{y_n^i,\varepsilon_n^i}v_i\Big\|_{\DR^{1,p}_b(\R^2,0)}\to0 ,\\
&\|u_n\|^p_{\DR^{1,p}_b(\Omega,0)}\to\|v_0\|^p_{\DR^{1,p}_b(\Omega,0)}
+\sum_{j=1}^k\|w_j\|^p_{\DR^{1,p}_b(\R^2,0)}+\sum_{i=1}^l\|v_i\|^p ,\\
&\widetilde J_\Omega(v_0)+\sum_{j=1}^k\widetilde J_{\R^2}(w_j)+\sum_{i=1}^lJ_{\R^2}(v_i)=c .
\end{split}
\end{equation*}
\end{thmletter}

It is well-known that results of the kind above, together with the classification of the entire
solutions of the critical equation in $\R^N$, recover compactness for certain ranges of the energy level $c$;
see e.g.\ \cite[Theorem 3.3]{MercuriPerera} and the references therein. The entire solutions
in question are the functions
\begin{equation}
\label{eq:AuTa}
u_{\varepsilon,y}(x)=\frac{c_{N,p}\,\varepsilon^{\frac{N-p}{p(p-1)}}}
{\big(\varepsilon^{\frac p{p-1}}+|x-y|^{\frac p{p-1}}\big)^{\frac{N-p}p}},
\qquad\varepsilon>0,\ y\in\R^N,
\end{equation}
$c_{N,p}$ chosen so that
$\int_{\R^N}|\nabla u_{\varepsilon,y}|^pdx=\int_{\R^N}u_{\varepsilon,y}^{p^*}dx=S_{N,p}^{N/p}$,
the extremals of the Sobolev inequality, where
\begin{equation*}
S_{N,p}\Big(\int_{\R^N}|u|^{p^*}dx\Big)^{\frac p{p^*}}\ \leq\ \int_{\R^N}|\nabla u|^p\,dx
\qquad\mbox{for every }u\in\DR^{1,p}(\R^N)
\end{equation*}
and $S_{N,p}$ is the optimal constant such that
\begin{equation*}
S_{N,p}=\inf\Big\{\int_{\R^N}|\nabla u|^p\,dx:\ u\in\DR^{1,p}(\R^N),\
\int_{\R^N}|u|^{p^*}dx=1\Big\} .
\end{equation*}
For a function attaining it the two integrals above coincide, and their common value is $S_{N,p}^{N/p}$ (Aubin~\cite{Aubin} and Talenti~\cite{Talenti}). For $p=2$
the classification is due to Caffarelli, Gidas and Spruck~\cite{CaffarelliGidasSpruck} (see also
Chen and Li~\cite{ChenLi}), after Gidas, Ni and Nirenberg~\cite{GidasNiNirenberg}; for the $p$-Laplacian, the uniqueness of the positive radial solution is
due to Guedda and V\'eron~\cite{GueddaVeronLocal}, and by Farina, Mercuri and
Willem~\cite[Proposition 1.3, pp.~4--5]{FarinaMercuriWillem} there are no other nontrivial radial
finite energy solutions, up to sign; the classification of positive solutions in $\DR^{1,p}(\R^N)$
was obtained for $\frac{2N}{N+2}\leq p<2$ by Damascelli, Merch\'an, Montoro and
Sciunzi~\cite{DamascelliMerchanMontoroSciunzi}, for $1<p<2$ by
V\'etois~\cite[Corollary 1.3, p.~151]{Vetois} and for $2<p<N$ by
Sciunzi~\cite[Theorem 1.1]{Sciunzi}.

\subsection{Remarks on nonexistence and compactness in higher dimension under axial symmetry}
\label{subsec:axial}

Let $N\geq3$, write $x=(x',x_N)\in\R^{N-1}\times\R$ and $r=|x'|$, and let $G=SO(N-1)\times\{1\}$
be the group of the rotations of $\R^N$ about the $x_N$-axis. We say that a function $u$ on a
$G$-invariant domain of $\R^N$ is \emph{axially symmetric} when $u(gx)=u(x)$ for every $g\in G$,
that is, when $u$ depends on $x'$ only through $|x'|$: $u(x)=U(r,x_N)$. We describe here the main
ideas by which the nonexistence and compactness results of this paper can be extended to such
solutions in higher dimension $N\geq3$, for $1<p<2$, as a result of a foliation argument which
reduces the problem to the weak unique continuation property of Theorem~\ref{thm:ucp}(i).

In cylindrical coordinates an axially symmetric solution $u=U(r,x_N)$ of
$-\Delta_pu=|u|^{p^*-2}u$ satisfies, on the leaf of the half-plane $\{r>0\}$, the weighted equation
\begin{equation*}
-\div\big(r^{N-2}|\nabla U|^{p-2}\nabla U\big)=r^{N-2}|U|^{p^*-2}U ,
\end{equation*}
where, acting on $U$, $\nabla$ and $\div$ are taken in the variables $(r,x_N)$. Expanding the
divergence, this is the planar $p$-Laplace equation with a first-order term,
\begin{equation*}
-\Delta_pU=|U|^{p^*-2}U+\frac{N-2}{r}\,|\nabla U|^{p-2}\partial_rU .
\end{equation*}
The weight $r^{N-2}$ is, up to a multiplicative constant, the $(N-2)$-dimensional measure of
$Gx=\{(y',x_N):\ |y'|=r\}$, namely the orbit under $G$ of a point $x$ with $|x'|=r$.
The regularity preliminaries we shall describe in Section~\ref{sec:reg} apply here, as the regularity
arguments do not depend on the dimension, that is,
$u\in C^{1,\alpha}_{{\textrm{loc}}}\cap W^{2,2}_{{\textrm{loc}}}$ in $\R^N$; hence $U$ inherits
the same regularity away from the axis, with $|D^2u|^2=|D^2U|^2+\frac{N-2}{r^2}\,(\partial_rU)^2$, and the weight $r^{N-2}$ of
the volume element is bounded above and below on compact subsets of $\{r>0\}$.

Away from the axis the reduced equation is of the type considered in Remark~\ref{rem:gradient}:
the first-order term is bounded by $\frac{N-2}{r_0}|\nabla U|^{p-1}$ on $\{r\geq r_0\}$, and the
nonlinearity $|U|^{p^*-2}U$ satisfies the growth condition \eqref{eq:f-growth} with
$C_M=M^{p^*-p}$, as in the proof of Theorem~\ref{thm:main}, since $p^*=\frac{Np}{N-p}>p$ in every
dimension. As the proof of Theorem~\ref{thm:ucp}(i) uses the equation only through its
non-divergence form, whose right-hand side is bounded on compact sets by a constant times
$|u|+|\nabla u|$, see \eqref{eq:route-key} in Section~\ref{sec:linear},
Theorem~\ref{thm:ucp}(i) holds for $U$ on the connected open set $\{r>r_0\}$, for every $r_0>0$.
We then proceed as follows. Fix $r_0>0$ and argue as in the proof of Theorem~\ref{thm:main}: the
extension by zero of $u$ across $\partial\R^N_+$ is an axially symmetric solution in $\R^N$,
which vanishes on the open set $\{r>r_0,\ x_N<0\}$, so that $U\equiv0$ on $\{r>r_0\}$ by
Theorem~\ref{thm:ucp}(i); since $r_0$ is arbitrary, $u\equiv0$.

Solutions can be extended by zero outside $\R^N_+$ to solutions in $\R^N$, by
Lemma~\ref{lem:normal} below, which holds in arbitrary dimension; see
\cite[Lemma 2.3, p.~472]{MercuriWillem}. Arguing as in the proofs of Theorems~\ref{thm:main}
and~\ref{thm:star}, it is easy to see that they hold for every $N\geq3$ and $1<p<2$ for axially
symmetric solutions: on the half-space, with the axis normal to its boundary, and on bounded
$C^2$ domains which are axially symmetric and star-shaped about a point of the axis, the ball in
particular, working in the latter case on the leaf $\{(r,x_N):\ r>0,\ (x',x_N)\in\Omega\}$,
which is connected. This extends the radial nonexistence of
\cite[Proposition 1.2, p.~4]{FarinaMercuriWillem} to axially symmetric solutions, for $1<p<2$,
and partially answers a question on non-radial sign-changing solutions raised in
\cite[p.~482]{MercuriWillem}.

We feel that symmetry is, to some extent, necessary for this approach to nonexistence. In fact,
restricting a solution to the planes $z=\mathrm{const}$, $x=(y,z)\in\R^2\times\R^{N-2}$, produces
no planar equation: with $v=u(\cdot,z)$ one has $|\nabla u|^2=|\nabla_yv|^2+|\nabla_zu|^2$ and
\begin{equation*}
\div_y\big(|\nabla u|^{p-2}\nabla_yv\big)=-|v|^{p^*-2}v-\div_z\big(|\nabla u|^{p-2}\nabla_zu\big),
\end{equation*}
where the weight $|\nabla u|^{p-2}$ depends on the derivatives of $u$ in the $z$ directions and,
once the divergences are expanded, the second derivatives $\partial_{zz}u$ and the mixed ones
$\partial_{yz}u$ appear, none of which is bounded by $C(|v|+|\nabla_yv|)$. This suggests that
the bound \eqref{eq:route-key} is not available for $v$, which the axial symmetry allows to
bypass: the directions orthogonal to the leaves then contribute only to the first-order term of the reduced
equation.

As a byproduct of nonexistence, one may then extend Theorem~\ref{thm:compact} to higher
dimensions provided the Palais--Smale sequences are axially symmetric. Preliminary calculations
show that concentration may occur only on the axis, so that half-space profiles can arise only
where the axis meets $\partial\Omega$; Theorem~\ref{thm:main}, in its axially symmetric form in
higher dimension and for $1<p<2$, would exclude them, and the standard bubbles in $\R^N$ would
remain the only ones responsible for the possible loss of compactness. \newline
 It is worth noting that if $\overline\Omega$ does not meet the symmetry
axis, no concentration occurs at all, and axially symmetric Palais--Smale sequences are compact.
This can be employed to obtain compactness at every level, and hence improved existence and multiplicity
results for critical problems in the spirit of Brezis and Nirenberg \cite{BrezisNirenberg}; see Clapp and Pacella~\cite[p.~587]{ClappPacella}, \cite[Theorem 2.2]{MercuriPacella} and ~\cite[Theorem 3.3]{MercuriPerera}.

\section{$W^{2,2}_{{\textrm{loc}}}$ regularity}
\label{sec:reg}

Throughout Sections~\ref{sec:reg}--\ref{sec:proofA}, $\Omega\subset\R^2$ is open, $1<p<2$,
$f:\R\to\R$ is continuous and satisfies \eqref{eq:f-growth} for every $M>0$, and a
solution of \eqref{eq:eq-f} is a weak solution $u\in C^{1,\alpha}_{{\textrm{loc}}}(\Omega)$ of
$-\Delta_pu=f(u)$ in $\Omega$, as in Theorem~\ref{thm:ucp}. We set
\begin{equation*}
Z=\{\nabla u=0\},\qquad \Sigma_0=Z\cap\{u=0\},\qquad e=\frac{\nabla u}{|\nabla u|}\ \mbox{ away from }Z ,
\end{equation*}
and, for $E$ with $\overline E$ a compact subset of $\Omega$, $M_E=\max_{\overline E}|u|$ and $C_E=C_{M_E}$, the constant of
\eqref{eq:f-growth} with $M=M_E$.

Some remarks on the available regularity theory for $p$-Laplace equations with a zeroth-order term $f(u)$ are useful in what is to follow. Bounded weak solutions of \eqref{eq:eq-f} are
$C^{1,\alpha}_{{\textrm{loc}}}$, by DiBenedetto~\cite{DiBenedetto} and
Tolksdorf~\cite{Tolksdorf}, which is the class in which
Theorem~\ref{thm:ucp} is stated; see also \cite[Theorem 1.4]{MercuriRieySciunzi}.
We will apply the weak form of the above theorem on nonexistence results, in which the
boundedness of $u$ can be justified by adapting an argument of
Trudinger~\cite{Trudinger}, see e.g.\ Peral~\cite[Appendix E]{Peral}.

We now recall how the existence of the second derivative can be justified. For $1<p<2$ second
derivatives exist in the $L^2_{{\textrm{loc}}}$ sense, $u\in W^{2,2}_{{\textrm{loc}}}(\Omega)$,
and this is all we shall use to rewrite the equation in a suitable non-divergence form. The identities below involving second derivatives hold almost everywhere;
moreover
\begin{equation}
\label{eq:hessian-on-Z}
D^2u=0\qquad\mbox{a.e.\ on }Z ,
\end{equation}
by Stampacchia's classical result, the derivatives of a Sobolev function vanishing almost
everywhere on each of its level sets (see e.g. Gilbarg and
Trudinger \cite{GilbargTrudinger}), applied to each $u_{x_i}$ on
$\{u_{x_i}=0\}$, whose intersection is $Z$.

The $W^{2,2}_{{\textrm{loc}}}$ regularity follows from
Tolksdorf~\cite{Tolksdorf}: for $p\leq2$ the regularity stated in (2.2.2) of~\cite{Tolksdorf} is
$W^{2,p}_{{\textrm{loc}}}$, but the weighted estimate~\cite[(2.2.7)]{Tolksdorf}, together
with our assumption $\nabla u\in L^\infty_{{\textrm{loc}}}$, gives $W^{2,2}_{{\textrm{loc}}}.$ In fact, the $p$-Laplacian satisfies the structure
conditions~\cite[(1.3)--(1.7), p.~127]{Tolksdorf} with $\kappa=0$, and $f(u)$ is bounded on
compact subsets of $\Omega$, so modulo translation and a scaling one is in the reference
situation of~\cite[\S2, p.~128]{Tolksdorf}. For $p<2$ his proof, pp.~130--132, is based on the
structure conditions and on the classical monotonicity inequality
\begin{equation}
\label{eq:monotone}
\big\langle|\xi|^{p-2}\xi-|\eta|^{p-2}\eta,\ \xi-\eta\big\rangle\ \geq\
c\,\big(1+|\xi|+|\eta|\big)^{p-2}|\xi-\eta|^2,\qquad\xi,\eta\in\R^2 ,
\end{equation}
with $c$ depending only on $p$; see also
Lindqvist~\cite[p.~100]{Lindqvist}. Since $p<2$, the weight on the right is bounded away
from zero wherever $\nabla u$ is bounded, and the difference quotients of $\nabla u$ are then
estimated in $L^2$ without the need of estimating $|\nabla u|$ from below.

Stronger regularity results on the stress field $U=|\nabla u|^{p-2}\nabla u$
rather than on $u$, are due to Damascelli and
Sciunzi~\cite[Corollary 2.2, p.~497]{DamascelliSciunzi} and, in a sharp form for every $p>1$, to
Cianchi and Maz'ya~\cite[(1.2), p.~569]{CianchiMazya}; we will not use them here though. We will do employ the
$C^1$ regularity assumption on $u,$ namely the continuity of the complex gradient $h=\partial_zu$, and for the
first-order Taylor expansion of $u$ in the proof of Theorem~\ref{thm:ucp}(ii).

\section{From $\Delta_pu+f(u)=0$ to a first-order system for the complex gradient}
\label{sec:linear}

In this section we write the equation \eqref{eq:eq-f} in non-divergence form, derive from it the
first-order system \eqref{eq:h-eq} for the complex gradient, and then remove from that system the
only term not involving the unknown.

Expanding the divergence away from the critical set
$Z$,
\begin{equation*}
\Delta_pu=|\nabla u|^{p-2}\big(\Delta u+(p-2)\langle D^2u\,e,e\rangle\big)
=|\nabla u|^{p-2}\,a_{ij}\partial_{ij}u ,\qquad a_{ij}=\delta_{ij}+(p-2)\,e_ie_j ,
\end{equation*}
and we set $a_{ij}=\delta_{ij}$ on $Z$: a symmetric measurable matrix, whose eigenvalues are $p-1$,
in the direction $e$, and $1$ away from $Z$, and both equal to $1$ on $Z$. Hence
\begin{equation*}
(p-1)|\xi|^2\ \leq\ a_{ij}\xi_i\xi_j\ \leq\ |\xi|^2\qquad
\mbox{for every }\xi\in\R^2 ,
\end{equation*}
almost everywhere in $\Omega$, so that the operator $a_{ij}\partial_{ij}$ is uniformly elliptic.

The equation then takes the non-divergence form
\begin{equation*}
a_{ij}\partial_{ij}u=-f(u)\,|\nabla u|^{2-p} ,
\end{equation*}
a non-homogeneous uniformly elliptic equation, whose right-hand side is bounded, using Young's
inequality, by a linear expression in the pair $(u,\nabla u)$:
\begin{equation}
\label{eq:route-key}
\big|a_{ij}\partial_{ij}u\big|\ \leq\ C_E\big(|u|+|\nabla u|\big)\qquad\mbox{a.e.\ in }E ,
\end{equation}
for every $E$ with $\overline E$ a compact subset of $\Omega$, with $C_E$ the constant $C_M$ of \eqref{eq:f-growth} for
$M=\max_{\overline E}|u|$. Both the expansion of $\Delta_pu$ and the non-divergence form are
obtained away from $Z$, where $e$ is defined, almost everywhere by the chain rule, since there $|\nabla u|^{p-2}\nabla u$ is a smooth function of $\nabla u\in W^{1,2}_{{\textrm{loc}}}$; \eqref{eq:route-key} holds
almost everywhere on $E$ because on $Z$ its left-hand side vanishes almost everywhere, as $D^2u$
does by \eqref{eq:hessian-on-Z}. Inequality \eqref{eq:route-key} is key to the approach we follow, as it will allow us to
use the classical theory available on Beltrami-type equations, which requires the coefficients of
the first-order system for the complex gradient of $u$ to be measurable and bounded.

We now write \eqref{eq:eq-f} as a first-order system for the complex gradient of $u$. With $z=x+iy$
and the Cauchy--Riemann operators $\partial_z=\frac12(\partial_x-i\partial_y)$,
$\partial_{\bar z}=\frac12(\partial_x+i\partial_y)$, we set
\begin{equation*}
h=\partial_zu=\tfrac12(u_x-iu_y),\qquad\mbox{so that}\qquad \partial_{\bar z}u=\bar h,\qquad |\nabla u|=2|h| ;
\end{equation*}
$h\in W^{1,2}_{{\textrm{loc}}}(\Omega)$ since $u\in W^{2,2}_{{\textrm{loc}}}(\Omega)$, and $h$ is
continuous since $u\in C^1$.
Away from $Z$, identifying $e$ with $\bar h/|h|$, as $\nabla u$ corresponds to $2\bar h$, we have
$\Delta u=4\partial_{\bar z}h$ and
$\langle D^2u\,e,e\rangle=2\,\mathrm{Re}\big((\partial_zh)\,\bar h/h\big)+2\partial_{\bar z}h$, so
\begin{equation*}
g=a_{ij}\partial_{ij}u=\Delta u+(p-2)\langle D^2u\,e,e\rangle
\end{equation*}
reads as
\begin{equation*}
g=2p\,\partial_{\bar z}h+(p-2)\Big((\partial_zh)\frac{\bar h}{h}+\overline{\partial_z h}\,\frac{h}{\bar h}\Big) .
\end{equation*}
This gives the first-order system for the complex gradient: almost everywhere in $\Omega$,
\begin{equation}
\label{eq:h-eq}
\partial_{\bar z}h=\mu\,\partial_zh+\nu\,\overline{\partial_z h}+\alpha\,h+\beta\,u ,
\end{equation}
with the Beltrami coefficients
\begin{equation}
\label{eq:munu}
\begin{split}
&\mu=\frac{2-p}{2p}\,\frac{\bar h}{h},\quad \nu=\frac{2-p}{2p}\,\frac{h}{\bar h}\ \ \mbox{away from }Z,
\qquad \mu=\nu=0\ \mbox{ on }Z,\\
&|\mu|+|\nu|\leq k,\qquad k=\frac{2-p}{p}<1 ,
\end{split}
\end{equation}
and
\begin{equation}
\label{eq:alphabeta}
\begin{split}
&\alpha=\frac{g\,\bar h}{2p\,(u^2+|h|^2)},\quad \beta=\frac{g\,u}{2p\,(u^2+|h|^2)}\ \ \mbox{away from }\Sigma_0,\\
&\alpha=\beta=0\ \mbox{ on }\Sigma_0 ,
\end{split}
\end{equation}
chosen so that $\alpha h+\beta u=g/2p$. Since
$|g|\leq C_E(|u|+2|h|)\leq\sqrt5\,C_E\,(u^2+|h|^2)^{1/2}$ by \eqref{eq:route-key}, it follows
that
\begin{equation}
\label{eq:ab-bound}
|\alpha|+|\beta|\leq\frac{\sqrt5\,C_E}{p}\qquad\mbox{on }E .
\end{equation}
On $Z$, where $\mu=\nu=0$, the identity \eqref{eq:h-eq}
holds as well, both sides vanishing almost everywhere: $\partial_{\bar z}h=\partial_z\partial_{\bar z}u$ and
$g=a_{ij}\partial_{ij}u$ vanish almost everywhere on $Z$ by \eqref{eq:hessian-on-Z}, and
$\alpha h+\beta u$ vanishes there as well, being $g/2p$ away from $\Sigma_0$ and $0$ on $\Sigma_0$.

Throughout the rest of this section $u$ is a solution of \eqref{eq:eq-f} and $h=\partial_zu$ its complex
gradient, which solves the first-order system \eqref{eq:h-eq}, that is
\begin{equation*}
\partial_{\bar z}h-\mu \partial_zh-\nu\overline{\partial_z h}=\alpha h+\beta u ,\qquad
|\mu|+|\nu|\leq k=\frac{2-p}p<1,\qquad \alpha,\beta\ \mbox{bounded}.
\end{equation*}
The left-hand side is a Beltrami operator and the right-hand side
is linear in $h$, except for $\beta u$, the only term not involving $h$.
For $f=0$ both terms are absent, $g$ being zero, and $|\partial_{\bar z}h|\leq k|\partial_zh|$, that is, $h$ is
quasiregular: $h=\mathcal F\circ\chi$ with $\mathcal F$ holomorphic and $\chi$ quasiconformal, and the
critical points of $u$ are the zeros of a holomorphic function, carried by $\chi$, hence isolated.
For $f\neq0$ the factorization is replaced by the representation
theorem of Bers and Nirenberg, $\mathcal F\circ\chi$ with $\mathcal F$ a holomorphic function times a
nonvanishing continuous factor, which holds for solutions of
systems with terms linear in the unknown but not with a term such as $\beta u$; we use it in the
form of Lemma~\ref{lem:representation} below.
We therefore devote this section to reducing \eqref{eq:h-eq} to a system without it.

\subsubsection*{Removing the term $\beta u$}

Let us set
\begin{equation*}
\mathcal R=h-vu .
\end{equation*}
A direct computation, carried out in the proof of Lemma~\ref{lem:remove} below, shows that in the
equation satisfied by $\mathcal R$ the coefficient of $u$ vanishes when
\begin{equation*}
\partial_{\bar z}v-\mu \partial_zv-\nu\overline{\partial_z v}=\beta+\alpha v-|v|^2+\mu v^2+\nu\bar v^2 .
\end{equation*}
Its left-hand side is $\mathcal Lv$, where $\mathcal L$ is the
Beltrami operator of \eqref{eq:h-eq}; its right-hand side is quadratic in $v$, through the terms $-|v|^2$, $\mu v^2$ and $\nu\bar v^2$.
With such a $v$ the equation for $\mathcal R$ has no term in $u$ left, and
Lemma~\ref{lem:representation} applies to it.

This change of unknown, and its construction, are strongly inspired by the work of
Alessandrini~\cite{Alessandrini2012}. To remove zeroth-order terms in the plane, positive
multipliers are used in~\cite{Alessandrini2012}, built on small disks by a contraction, a classical
method going back, to our recollection,
to Bers, John and Schechter~\cite[p.~260]{BersJohnSchechter}; see the introduction
of~\cite{Alessandrini2012}. Here the multiplier is
replaced by its logarithmic derivative: if $\phi>0$ solved \eqref{eq:h-eq} in place of $u$, that is
$\mathcal L(\partial_z\phi)=\alpha\,\partial_z\phi+\beta\phi$, then
\begin{equation*}
v=\frac{\partial_z\phi}{\phi}\ \mbox{ would solve }\ \mathcal Lv=\beta+\alpha v-|v|^2+\mu v^2+\nu\bar v^2,
\qquad \mathcal R=\phi\,\partial_z\Big(\frac u\phi\Big) ,
\end{equation*}
mirroring the elementary one variable case, where the logarithmic derivative of a positive
solution of a linear second-order equation solves a Riccati equation. Solving the equation for $v$
directly (see Lemma~\ref{lem:remove} below) is a way to bypass the use of $\phi$: no positive
solution has to be constructed, and $v$ need not be a logarithmic derivative.

Such a $v$, with $|v|\leq1$, exists around every point of $\Omega$, on disks of sufficiently small
radius. In fact it is constructed by writing the equation for $v$,
\begin{equation*}
\mathcal Lv=\mathcal Q(v),\qquad \mathcal Q(v)=\beta+\alpha v-|v|^2+\mu v^2+\nu\bar v^2 ,
\end{equation*}
as the fixed point problem $v=\mathcal L^{-1}\mathcal Q(v)$, in which the whole nonlinearity sits in
$\mathcal Q$, quadratic in $v$, and $\mathcal L^{-1}$ is a right inverse of $\mathcal L$ on a disk. Lemma~\ref{lem:T} below solves $\mathcal Lv=\psi$ on a disk $B_R$, for any measurable $\mu,\nu$ with
$|\mu|+|\nu|\leq k<1$, by a solution $v=\mathcal L^{-1}\psi$ depending (real-)linearly on $\psi$ and
satisfying $\|v\|_\infty\leq C_0R\|\psi\|_\infty$, with $C_0$ depending only on $k$: as an operator
on $L^\infty(B_R)$, $\mathcal L^{-1}$ has norm at most $C_0R$. On the unit ball of
$L^\infty(B_R)$ the map $\mathcal Q$ is bounded and Lipschitz, with constants depending only
on the bounds for $\alpha,\beta$ and on $k$; the factor $R$ therefore makes
$v\mapsto \mathcal L^{-1}\mathcal Q(v)$ a contraction of that ball for $R$ small enough, and its fixed point
produces the desired $v$.

The Beltrami operator is a perturbation of $\partial_{\bar z}$: $\mathcal L=\partial_{\bar z}-\mu\partial_z-\nu\overline{\partial_z}$
with $|\mu|+|\nu|\leq k<1$. The Cauchy transform $\mathcal C$ inverts $\partial_{\bar z}$, being the
convolution with its fundamental solution $\frac1{\pi z}$, and the Beurling transform
$\mathcal B=\partial_z\mathcal C$ expresses $\partial_z(\mathcal C\omega)$ in terms of
$\partial_{\bar z}(\mathcal C\omega)=\omega$. Looking for $v=\mathcal C\omega$, the unknown becomes
$\omega=\partial_{\bar z}v$, and the equation $\mathcal Lv=\psi$ becomes the integral equation
$\omega-\mu\mathcal B\omega-\nu\overline{\mathcal B\omega}=\psi$.

We now start building $\mathcal L^{-1}$, whose construction is based on the
following properties of the Cauchy and Beurling transforms.

\begin{prop}[Cauchy and Beurling transforms]
\label{prop:transforms}
Let $\mathcal C$ be the Cauchy transform and $\mathcal B$ the Beurling transform,
\begin{equation*}
\mathcal C\omega(z)=-\frac1\pi\int_{\mathbb C}\frac{\omega(\zeta)}{\zeta-z}\,dA(\zeta),
\qquad
\mathcal B\omega(z)=-\frac1\pi\,\lim_{\delta\to0}\int_{|\zeta-z|>\delta}\frac{\omega(\zeta)}{(\zeta-z)^2}\,dA(\zeta),
\end{equation*}
where $dA$ is the Lebesgue measure of $\mathbb C=\R^2$; the first integral is taken for $\omega$
with compact support, and the limit in the second in $L^q(\mathbb C)$ for
$\omega\in L^q(\mathbb C)$, $1<q<\infty$.
\begin{enumerate}
\item[(i)] $\mathcal B$ is bounded on $L^q(\mathbb C)$ for $1<q<\infty$.
\item[(ii)] Its operator norm $\|\mathcal B\|_q$ on $L^q(\mathbb C)$ satisfies $\lim_{q\to2}\|\mathcal B\|_q=1$.
\item[(iii)] For $\omega\in L^q(\mathbb C)$ with compact support and $q>2$: $\mathcal C\omega$ is continuous,
$\mathcal C\omega\in W^{1,q}_{{\textrm{loc}}}(\mathbb C)$, and $\partial_{\bar z}(\mathcal C\omega)=\omega$, $\partial_z(\mathcal C\omega)=\mathcal B\omega$.
\end{enumerate}
\end{prop}

The use of these transforms in $L^q$, $q>2$, for Beltrami equations with measurable coefficients is
due to Bojarski~\cite{Bojarski}, and is based on (ii)~\cite[(1.19)]{Bojarski}: for a
coefficient bounded by $k<1$ it gives $q>2$ with $k\|\mathcal B\|_q<1$, which is the only bound on
$\mathcal B$ that we use in Lemma~\ref{lem:T}. See also
Ahlfors~\cite[pp.~51--66]{Ahlfors}.\footnote{In Ahlfors~\cite{Ahlfors}, $\mathcal B$ is the operator $T$ and
$\mathcal C$ differs from the operator $P$ by an additive constant, the normalization at the origin;
(i) and (ii) are the Calder\'on--Zygmund inequality $\|Th\|_q\leq C_q\|h\|_q$, with $C_q\to1$ as
$q\to2$, the limit coming from the Riesz--Thorin convexity theorem; and (iii) is Lemmas 1 and 3
of~\cite{Ahlfors}.}

We are now in a position to invert the Beltrami operator $\mathcal L$ on a disk; as anticipated,
we denote this right inverse by $\mathcal L^{-1}$.

\begin{lemma}[Solvability of the Beltrami equation on a disk]
\label{lem:T}
Let $0\leq k<1$. There are $q>2$ and $C_0>0$, depending only on $k$, with the
following property. Let $B_R\subset\mathbb C$ be a disk, let $\mu,\nu$ be measurable on $B_R$ with
$|\mu|+|\nu|\leq k$, and let
\begin{equation*}
\mathcal Lv=\partial_{\bar z}v-\mu \partial_zv-\nu\overline{\partial_z v}
\end{equation*}
be the corresponding Beltrami operator. Then, for every $\psi\in L^\infty(B_R)$, the equation
$\mathcal Lv=\psi$ has a solution $v\in W^{1,q}(B_R)\cap C(\overline{B_R})$, which depends
real-linearly on $\psi$ and satisfies
\begin{equation}
\label{eq:T}
\|v\|_{L^\infty(B_R)}\leq C_0R\,\|\psi\|_{L^\infty(B_R)} .
\end{equation}
We write $v=\mathcal L^{-1}\psi$: a right inverse of $\mathcal L$, of norm at most $C_0R$ on
$L^\infty(B_R)$, only real-linear because of the term $\nu\overline{\partial_z v}$.
\end{lemma}

\begin{proof}
By Proposition~\ref{prop:transforms}(ii) we fix $q>2$ so close to $2$ that
$k\|\mathcal B\|_q<1$; $q$ and the constants below depend on $k$ only. Extend $\mu$, $\nu$ and $\psi$ by
zero outside $B_R$ and look for $\mathcal L^{-1}\psi$ in the form $\mathcal C\omega$ with $\omega\in L^q(\mathbb C)$. Since
$\partial_{\bar z}(\mathcal C\omega)=\omega$ and $\partial_z(\mathcal C\omega)=\mathcal B\omega$ (Proposition~\ref{prop:transforms}(iii)),
the equation for $\mathcal C\omega$ becomes an integral equation for $\omega$:
\begin{equation}
\label{eq:neumann}
\omega-\mu \mathcal B\omega-\nu\overline{\mathcal B\omega}=\psi .
\end{equation}
Since $|\mu \mathcal B\omega+\nu\overline{\mathcal B\omega}|\leq k|\mathcal B\omega|$, the operator
$\omega\mapsto\mu \mathcal B\omega+\nu\overline{\mathcal B\omega}$ has norm at most $k\|\mathcal B\|_q<1$ on
$L^q(\mathbb C)$, and \eqref{eq:neumann} has exactly one solution $\omega\in L^q(\mathbb C)$,
given by a Neumann series (see e.g.\ Reed and Simon~\cite{ReedSimon}), with
\begin{equation}
\label{eq:omega-bound}
\|\omega\|_q\leq(1-k\|\mathcal B\|_q)^{-1}\|\psi\|_q .
\end{equation}
Outside $B_R$,
where $\mu$, $\nu$ and $\psi$ vanish, the equation reads $\omega=0$: the solution is supported in
$\overline{B_R}$, and $\mathcal C\omega$ is continuous and in $W^{1,q}(B_R)$ by
Proposition~\ref{prop:transforms}(iii). We set $v=\mathcal L^{-1}\psi=\mathcal C\omega$ on $B_R$; it solves $\mathcal Lv=\psi$, and depends
real-linearly on $\psi$ because $\omega$ does.

To justify \eqref{eq:T}, we first estimate $\psi$ in \eqref{eq:omega-bound}:
since $\psi$ vanishes outside $B_R$,
\begin{equation}
\label{eq:omega-infty}
\|\omega\|_q\leq(1-k\|\mathcal B\|_q)^{-1}\|\psi\|_q\leq(1-k\|\mathcal B\|_q)^{-1}(\pi R^2)^{1/q}\|\psi\|_\infty .
\end{equation}
On the other hand, by H\"older's inequality with $q'=\frac q{q-1}<2$, for $z\in B_R$, since $\omega$
vanishes outside $B_R$ and $B_R\subset B_{2R}(z)$,
\begin{equation}
\label{eq:P-bound}
\begin{split}
|\mathcal C\omega(z)|&\leq\frac1\pi\|\omega\|_q\Big(\int_{B_{2R}(z)}|\zeta-z|^{-q'}dA\Big)^{1/q'}
=c(q)\,R^{1-\frac2q}\,\|\omega\|_q ,\\
c(q)&=\frac{2^{1-\frac2q}}\pi\Big(\frac{2\pi}{2-q'}\Big)^{1/q'}.
\end{split}
\end{equation}
Combining \eqref{eq:omega-infty} and \eqref{eq:P-bound} gives \eqref{eq:T}, with
$C_0=\pi^{1/q}c(q)(1-k\|\mathcal B\|_q)^{-1}$, which depends only on $k$, as $q$ does.
\end{proof}

\begin{lemma}[Cancelling $\beta u$ in \eqref{eq:h-eq} via a fixed point of $\mathcal L^{-1}\mathcal Q$]
\label{lem:remove}
Let $u$ be a solution of \eqref{eq:eq-f}, let $h=\partial_zu$, let $\mu,\nu,\alpha,\beta$ be as in
\eqref{eq:munu}--\eqref{eq:alphabeta}, and let $q$ be the exponent of Lemma~\ref{lem:T} for
$k=\frac{2-p}p$. For every $z_0\in\Omega$ there exists $R_0>0$ such that, for
every $0<R\leq R_0$, writing $B_R=B_R(z_0)$, there exists a function $v\in W^{1,q}(B_R)\cap C(\overline{B_R})$, with
$|v|\leq1$, such that $\mathcal R=h-vu$ solves
\begin{equation}
\label{eq:V-hom}
\partial_{\bar z}\mathcal R=\mu \partial_z\mathcal R+\nu\overline{\partial_z \mathcal R}+A\mathcal R+B\overline{\mathcal R}\ \ \mbox{a.e.\ in }B_R,
\qquad A=\alpha+\mu v,\quad B=\nu\bar v-v ,
\end{equation}
with $A$ and $B$ bounded on $B_R$. Moreover
\begin{equation}
\label{eq:V-reg}
\mathcal R\in W^{1,2}(B_R)\cap L^\infty(B_R) ,
\end{equation}
and if $u(z_0)=0$ and $u\not\equiv0$ on every disk $B_r(z_0)$, then $\mathcal R\not\equiv0$ in $B_r(z_0)$
for every $0<r\leq R$.
\end{lemma}

\begin{proof}
Fix $z_0\in\Omega$ and a disk $B_R=B_R(z_0)$ with $\overline{B_R}\subset\Omega$, to be shrunk in
Step 2, and let $\Lambda$ be a bound for $|\alpha|+|\beta|$ on $B_R$, from \eqref{eq:ab-bound}.

\emph{Step 1: the term $\beta u$ disappears if $v$ solves a quadratic equation.} Let $v\in W^{1,q}(B_R)$ be bounded. Differentiating
$\mathcal R=h-vu$ and using $\partial_zu=h$, $\partial_{\bar z}u=\bar h$, we obtain
\begin{equation*}
\partial_{\bar z}\mathcal R-\mu \partial_z\mathcal R-\nu\overline{\partial_z \mathcal R}
=\big(\partial_{\bar z}h-\mu \partial_zh-\nu\overline{\partial_z h}\big)
-\big(\partial_{\bar z}v-\mu \partial_zv-\nu\overline{\partial_z v}\big)u
-v\bar h+\mu v h+\nu\bar v\bar h .
\end{equation*}
The first bracket is $\alpha h+\beta u$ by \eqref{eq:h-eq}; writing $h=\mathcal R+vu$ and
$\bar h=\overline{\mathcal R}+\bar vu$ in the remaining terms and collecting, we obtain almost everywhere in $B_R$
\begin{equation}
\label{eq:riccati}
\begin{split}
\partial_{\bar z}\mathcal R-\mu \partial_z\mathcal R-\nu\overline{\partial_z \mathcal R}
&=(\alpha+\mu v)\,\mathcal R+(\nu\bar v-v)\,\overline{\mathcal R}\\
&\quad+\Big[\beta+\alpha v-|v|^2+\mu v^2+\nu\bar v^2
-\big(\partial_{\bar z}v-\mu \partial_zv-\nu\overline{\partial_z v}\big)\Big]\,u .
\end{split}
\end{equation}
Hence $\mathcal R$ solves the homogeneous system \eqref{eq:V-hom} as soon as $v$ solves
\begin{equation}
\label{eq:riccati-eq}
\partial_{\bar z}v-\mu \partial_zv-\nu\overline{\partial_z v}=\mathcal Q(v)\ \ \mbox{a.e.\ in }B_R,
\qquad
\mathcal Q(v)=\beta+\alpha v-|v|^2+\mu v^2+\nu\bar v^2 ,
\end{equation}
and for $|v|\leq1$ the coefficients $A$ and $B$ are bounded on $B_R$, since $|\alpha|\leq\Lambda$
and $|\mu|+|\nu|<1$.

\emph{Step 2: Solution to $\mathcal Lv=\mathcal Q(v)$ by a contraction argument on a small disk.} Let $\mathcal L^{-1}$ and $C_0$ be as in
Lemma~\ref{lem:T} on $B_R$, for $k=\frac{2-p}p$ and the coefficients $\mu,\nu$ of
\eqref{eq:munu}. A fixed point of $v\mapsto\mathcal L^{-1}\mathcal Q(v)$ solves \eqref{eq:riccati-eq},
and since by Lemma~\ref{lem:T} the norm of $\mathcal L^{-1}$ is at most $C_0R$, this map is a
contraction for $R$ small.
Indeed, on the unit ball $X$ of $L^\infty(B_R)$ the map $\mathcal Q$ is bounded and Lipschitz, with
constants depending only on $\Lambda$, since $|\alpha|+|\beta|\leq\Lambda$ on $B_R$ and
$|\mu|+|\nu|<1$; so, by \eqref{eq:T}, $v\mapsto\mathcal L^{-1}\mathcal Q(v)$ sends $X$ into itself and
contracts distances once $R$ is small enough. The fixed point produced by the classical Banach-Caccioppoli lemma,
$v=\mathcal L^{-1}\mathcal Q(v)$, lies in $W^{1,q}(B_R)\cap C(\overline{B_R})$, satisfies $|v|\leq1$
and solves \eqref{eq:riccati-eq}. The regularity \eqref{eq:V-reg} then holds by construction:
$h\in W^{1,2}(B_R)$ is continuous and $vu\in W^{1,q}(B_R)\subset W^{1,2}(B_R)$ is bounded.

\emph{Step 3: $\mathcal R$ vanishes on a disk only if $u$ does.} If $\mathcal R\equiv0$ on $B_r(z_0)$, then $\partial_zu=vu$ there, that is,
$\nabla u=2\big(\mathrm{Re}\,v,-\mathrm{Im}\,v\big)u$ and $|\nabla u|\leq2|u|$. Fix a unit vector
$\theta$ and let $\varphi(t)=u(z_0+t\theta)$ for $0\leq t<r$. Then, since $u(z_0)=0$, it holds
that $\varphi(0)=0$, and
\begin{equation*}
|\varphi'(t)|=\big|\nabla u(z_0+t\theta)\cdot\theta\big|\leq\big|\nabla u(z_0+t\theta)\big|
\leq2\,|\varphi(t)| ,
\end{equation*}
so $\varphi\equiv0$ by the fundamental theorem of calculus (as for Gronwall's general inequality). Every point of $B_r(z_0)$ lies on one of these
segments, whence $u\equiv0$ on $B_r(z_0)$.
\end{proof}

\section{The perturbed complex gradient of $u$ has isolated zeros}
\label{sec:representation}

By Lemma~\ref{lem:remove}, $\mathcal R=h-vu$ solves the homogeneous system \eqref{eq:V-hom} on a
disk $B_R=B_R(z_0)$, with $|\mu|+|\nu|\leq k<1$ and $A,B$ bounded, and lies in
$W^{1,2}(B_R)\cap L^\infty(B_R)$. It is therefore a generalized solution of
\eqref{eq:V-hom} in the sense of Bojarski~\cite{Bojarski}, and the representation theorem
of Bers and Nirenberg~\cite{BersNirenberg}, in Bojarski's form~\cite[Theorem 4.4,
p.~477]{Bojarski}, applies to it: $\mathcal R=e^{\sigma}\,(\mathcal G\circ\chi)$, with $\sigma$
continuous, $\mathcal G$ holomorphic and $\chi$ a homeomorphism. The zeros of $\mathcal R$ are then
those of a holomorphic function, carried by $\chi$, hence isolated unless $\mathcal R$ vanishes
identically; see Bojarski~\cite[p.~478]{Bojarski} and Bers, John and
Schechter~\cite[\S\S6.3--6.4, pp.~259--261]{BersJohnSchechter}. The next
lemma states this in the form used in the following two sections, with the factor $e^{\sigma}$
absorbed into $\mathcal F=e^{\sigma\circ\chi^{-1}}\mathcal G$.

A homeomorphism $\Psi$ between domains of $\mathbb C$ is \emph{$K$-quasiconformal}, $K\geq1$, when
it is orientation-preserving, lies in $W^{1,2}_{{\textrm{loc}}}$ and satisfies almost everywhere
\begin{equation*}
|\partial_z\Psi|+|\partial_{\bar z}\Psi|\ \leq\ K\big(|\partial_z\Psi|-|\partial_{\bar z}\Psi|\big) ,
\qquad\mbox{equivalently,}\qquad
|\partial_{\bar z}\Psi|\leq\tfrac{K-1}{K+1}|\partial_z\Psi| .
\end{equation*}
Note that, where $\Psi$ is differentiable, its derivative along the unit direction $e^{i\theta}$ is
$e^{i\theta}\partial_z\Psi+e^{-i\theta}\partial_{\bar z}\Psi$, and
\begin{equation*}
|\partial_z\Psi|+|\partial_{\bar z}\Psi|\ \geq\
\big|e^{i\theta}\,\partial_z\Psi+e^{-i\theta}\,\partial_{\bar z}\Psi\big|\ \geq\
|\partial_z\Psi|-|\partial_{\bar z}\Psi| ,
\end{equation*}
both bounds being attained as $\theta$ varies; hence the condition says that $\Psi$ stretches at
most $K$ times more in one direction than in another. See Lehto and Virtanen~\cite[p.~168]{LehtoVirtanen} and also
Astala, Iwaniec and Martin~\cite[p.~24]{AstalaIwaniecMartin}.

\begin{lemma}[Representation of $\mathcal R$ and isolatedness of its zeros]
\label{lem:representation}
Let $u$, $v$, $\mathcal R=h-vu$ and $B_R=B_R(z_0)$ be as in Lemma~\ref{lem:remove}, and let
$k=\frac{2-p}p$ and $K=\frac{1+k}{1-k}$. Then there are a $K$-quasiconformal homeomorphism
$\chi$ of $\mathbb C$ onto $\mathbb C$ and a continuous function $\mathcal F$ on $\chi(B_R)$ such
that
\begin{equation}
\label{eq:representation}
\mathcal R=\mathcal F\circ\chi\qquad\mbox{in }B_R .
\end{equation}
If moreover $u(z_0)=0$ and $u\not\equiv0$ on every disk $B_r(z_0)$, then near each zero $\zeta_1$ of
$\mathcal F$ it holds that $\mathcal F(\zeta)=(\zeta-\zeta_1)^n\Theta(\zeta)$, for some integer
$n\geq1$ and some continuous $\Theta$ which does not vanish; in particular, the zeros of
$\mathcal F$ are isolated, as well as those of $\mathcal R$ in $B_R$.
\end{lemma}

\begin{proof}
On $B_R$ the coefficients of \eqref{eq:V-hom} satisfy $|\mu|+|\nu|\leq k<1$, while $A$ and $B$
are bounded, hence in $L^{s}(B_R)$ for every $s$; by \eqref{eq:V-reg}, $\mathcal R$ is a generalized
solution there. Bojarski's theorem~\cite[Theorem 4.4]{Bojarski} therefore gives
$\mathcal R=e^{\sigma}\,(\mathcal G\circ\chi)$ on $B_R$, with $\sigma$ continuous on
$\overline{B_R}$, $\mathcal G$ holomorphic on $\chi(B_R)$, and $\chi$ the normalized
solution of the associated Beltrami equation~\cite[(3.5), p.~464]{Bojarski}, whose coefficient is
extended by $0$ outside $B_R$: a homeomorphism of $\mathbb C$ onto $\mathbb C$ whose Beltrami coefficient is
bounded by $k$, that is, $K$-quasiconformal with $K=\frac{1+k}{1-k}$. Hence
$\mathcal R=\mathcal F\circ\chi$ with
\begin{equation*}
\mathcal F(\zeta)=e^{\sigma(\chi^{-1}(\zeta))}\,\mathcal G(\zeta),\qquad \zeta\in\chi(B_R) .
\end{equation*}
Being the product of a continuous factor which never vanishes and a holomorphic one, $\mathcal F$
is continuous on $\chi(B_R)$ and has the same zeros of $\mathcal G$.

If $u(z_0)=0$ and $u\not\equiv0$ on every disk $B_r(z_0)$, then $\mathcal R\not\equiv0$ in $B_R$ by
Lemma~\ref{lem:remove}, hence $\mathcal G\not\equiv0$ on the connected open set $\chi(B_R)$.
Its zeros are therefore isolated, and near each of them
\begin{equation*}
\mathcal G(\zeta)=(\zeta-\zeta_1)^n\,g(\zeta),\qquad n\geq1\ \mbox{an integer},\quad
g\ \mbox{holomorphic},\quad g(\zeta_1)\neq0
\end{equation*}
(see e.g.\ Rudin~\cite[Theorem 10.18]{Rudin}). On a disk around $\zeta_1$ where $g$ does not
vanish, the function $\Theta=e^{\sigma\circ\chi^{-1}}\,g$ is continuous and does not vanish either,
and
\begin{equation*}
\mathcal F(\zeta)=e^{\sigma(\chi^{-1}(\zeta))}\,\mathcal G(\zeta)=(\zeta-\zeta_1)^n\,\Theta(\zeta) .
\end{equation*}
Finally observe that $\chi^{-1}$ maps the zeros of $\mathcal F$ bijectively onto those of
$\mathcal R$, which are therefore isolated too. This concludes the proof.
\end{proof}

We conclude this section with a tool which turns out to be useful for the strong form of the
unique continuation property, Theorem~\ref{thm:ucp}(ii), as it prevents $\mathcal R$ from
vanishing at a zero $z_0$ faster than a power of $|z-z_0|$. By the factorization of
$\mathcal F$ at $\chi(z_0)$, near $z_0$ the modulus of $\mathcal R$ is bounded from below by a
multiple of $|\chi(z)-\chi(z_0)|^n$. It is therefore enough that $\chi$ contract distances no
faster than a power:
\begin{equation*}
|\chi(z)-\chi(z_0)|\ \geq\ c\,|z-z_0|^{K}
\qquad\Longrightarrow\qquad
|\mathcal R(z)|\ \geq\ c'\,|z-z_0|^{nK} .
\end{equation*}
The distortion bound on the left is provided by next lemma, which is based on the H\"older
continuity of quasiconformal mappings, due to Mori~\cite{Mori}.

\begin{lemma}[H\"older distortion]
\label{lem:holder}
Let $\chi$ be a $K$-quasiconformal homeomorphism of $\mathbb C$ onto $\mathbb C$, $K\geq1$. Then
for every $R>0$ there is $c>0$ such that
\begin{equation}
\label{eq:holder}
|\chi(z)-\chi(z')|\ \geq\ c\,|z-z'|^{K}\qquad\mbox{for every }z,z'\in\overline{B_R(0)} .
\end{equation}
\end{lemma}

\begin{proof}
The inverse of a $K$-quasiconformal mapping is $K$-quasiconformal, see Lehto and
Virtanen~\cite[p.~17]{LehtoVirtanen}, so $\chi^{-1}$ is a $K$-quasiconformal mapping of
$\mathbb C$ onto $\mathbb C$. By Mori's theorem, in the form given by Lehto and
Virtanen~\cite[Theorem 4.3]{LehtoVirtanen}, a $K$-quasiconformal mapping of a
domain is uniformly H\"older continuous with exponent $\frac1K$ on every compact subset of that
domain. Hence, on the compact set $\chi\big(\overline{B_R(0)}\big)$, the mapping $\chi^{-1}$ is
H\"older continuous with exponent $\frac1K$: there is $C>0$ such that, for every
$z,z'\in\overline{B_R(0)}$,
\begin{equation*}
|z-z'|=\big|\chi^{-1}(\chi(z))-\chi^{-1}(\chi(z'))\big|\ \leq\ C\,|\chi(z)-\chi(z')|^{1/K} .
\end{equation*}
This concludes the proof.
\end{proof}

\textcolor{red}{\begin{corollary}
\label{cor:isolated}
Let $u\not\equiv0$ be a solution as in Theorem~\ref{thm:ucp}. Then $\Sigma_0=Z\cap\{u=0\}$ consists of
isolated points.
\end{corollary}}

\section{Proof of Theorem A}
\label{sec:proofA}

Both items of Theorem~\ref{thm:ucp} follow from the representation of $\mathcal R=h-vu$ in
Lemma~\ref{lem:representation}. For (i) we use the fact that where $u$ vanishes identically, so do $h=\partial_zu$ and
$\mathcal R$, while the zeros of $\mathcal R$ are isolated near every point where $u$ vanishes
without vanishing identically. For (ii) we note that $\mathcal R$ does not vanish faster than a power of the
distance, by Lemma~\ref{lem:holder}, and a Caccioppoli inequality will help transfer this bound over our original solution
$u$.

\begin{proof}[Proof of Theorem~\ref{thm:ucp}(i)]
Let $u$ be a solution of \eqref{eq:eq-f} on a connected $\Omega$, vanishing on a nonempty open
subset of $\Omega$, and let
\begin{equation*}
\Omega_0=\{z\in\Omega:\ u\equiv0\ \mbox{on some disk around }z\}
\end{equation*}
be the interior of $u^{-1}(0)$. $\Omega_0$ is open by its definition and nonempty by hypothesis. It will be shown that $\Omega_0$ is closed in
$\Omega$, hence $\Omega_0=\Omega$, that is, $u\equiv0$ in $\Omega$.

 Let $z_0\in\Omega$ be the limit of
a sequence $(z_n)_{n\in\mathbb N}\subset\Omega_0$; then $u(z_0)=0$, since $u$ is continuous. Suppose
$z_0\notin\Omega_0$, so that $u\not\equiv0$ on every disk around $z_0$. Let $B_R(z_0)$ and $v$ be as
in Lemma~\ref{lem:remove} for the point $z_0$, and let $\mathcal R=h-vu$ be as in that lemma. By
Lemma~\ref{lem:representation} the zeros of $\mathcal R$ in $B_R(z_0)$ are isolated. On the other
hand, for $n$ large, $z_n\in B_R(z_0)$, and $u$ vanishes identically on a disk
$D\subset B_R(z_0)$ around $z_n$: there $h=\partial_zu$ vanishes as well, and so does
$\mathcal R=h-vu$. Every point of $D$ is then a zero of $\mathcal R$ which is not isolated, a
contradiction. Hence $z_0\in\Omega_0$, and $\Omega_0$ is closed in $\Omega$.
\end{proof}

Part (ii) of Theorem~\ref{thm:ucp} is a strong unique continuation principle, and the proof yields
\eqref{eq:finite-order} at each point, with $m=m(u,z_0)$. For instance a pointwise condition
$|u(z)|=O(|z-z_0|^j)$ for every $j$ implies that $u$ vanishes to infinite order at $z_0$, true for
$e^{-1/|z-z_0|},$ but not for a nonzero polynomial.

\begin{proof}[Proof of Theorem~\ref{thm:ucp}(ii)]
Let $u\not\equiv0$ and let $z_0\in\Omega$. Our proof here \textcolor{red}{breaks into several steps, and
prove that} \eqref{eq:finite-order} holds with an explicit exponent.

\emph{Step 1: the two cases where either $u$ or $\nabla u$ does not vanish at $z_0$.} If $u(z_0)\neq 0$, then
$|u|\geq\frac12|u(z_0)|$ on some disk around $z_0$, so that
\begin{equation*}
\int_{B_\rho(z_0)}|u|^p\,dx\ \geq\ c\,\rho^2\qquad\mbox{for }\rho\ \mbox{small} ,
\end{equation*}
which is \eqref{eq:finite-order} with $m=2$.
In the other case, since $u$ is differentiable and vanishes at $z_0$, by $u(z)=\nabla u(z_0)\cdot(z-z_0)+o(|z-z_0|)$, it follows that, for small $\rho$, $|u|\geq c\,\rho$ on the disk of radius
$\frac\rho4$ centred at distance $\frac\rho2$ from $z_0$ in the direction of $\nabla u(z_0).$ Since that disc is
contained in $B_\rho(z_0)$, integrating we obtain
\begin{equation*}
\int_{B_\rho(z_0)}|u|^p\,dx\ \geq\ c'\,\rho^{p+2} ,
\end{equation*}
which is \eqref{eq:finite-order} with $m=p+2$. We focus now on the case $z_0\in\Sigma_0$. We recall that
\[ Z=\{\nabla u=0\},\qquad \Sigma_0=Z\cap\{u=0\}.\]

\emph{Step 2: $z_0\in\Sigma_0$, and a bound from below for $\mathcal R$.} By
Theorem~\ref{thm:ucp}(i), $u\not\equiv0$ on every disk around $z_0$. Let $B_R(z_0)$ and $v$ be
as in Lemma~\ref{lem:remove} for the point $z_0$, and let $\mathcal R=h-vu$ be as in that lemma; $\mathcal R(z_0)=h(z_0)-v(z_0)u(z_0)=0$. By
Lemma~\ref{lem:representation}, $\mathcal R=\mathcal F\circ\chi$ on $B_R(z_0)$, with $\chi$
$K$-quasiconformal for $K=\frac1{p-1}$, and, with $\zeta_0=\chi(z_0)$, there are an integer
$n\geq1$ and $\sigma>0$ such that $\mathcal F(\zeta)=(\zeta-\zeta_0)^n\Theta(\zeta)$, with $\Theta$
continuous and nonvanishing on $\overline{B_\sigma(\zeta_0)}\subset\chi(B_R(z_0))$. Put
$c_1=\min_{\overline{B_\sigma(\zeta_0)}}|\Theta|>0$ and fix $\rho_0>0$ with
$\overline{B_{\rho_0}(z_0)}\subset\chi^{-1}\big(B_\sigma(\zeta_0)\big)$. By
Lemma~\ref{lem:holder}, applied on a disk centred at $0$ containing
$\overline{B_{\rho_0}(z_0)}$,
\begin{equation*}
|\mathcal R(z)|\ \geq\ c_1\,|\chi(z)-\zeta_0|^n\ \geq\ c'\,|z-z_0|^{nK}
\qquad\mbox{on }B_{\rho_0}(z_0) ,
\end{equation*}
so that, for $0<\rho<\rho_0$,
\begin{equation}
\label{eq:R-below}
\int_{B_\rho(z_0)}|\mathcal R|^p\,dx\ \geq\ c''\,\rho^{nKp+2} .
\end{equation}

\emph{Step 3: the Caccioppoli inequality, estimating $\nabla u$ by $u$.} Shrinking $\rho_0$ further if necessary,
assume $\overline{B_{2\rho_0}(z_0)}\subset\Omega$ and $2\rho_0\leq1$, and let $C_E$ be the constant of
\eqref{eq:f-growth} for $E=B_{2\rho_0}(z_0)$. For $0<\rho<\rho_0$ let
$\eta\in C_c^\infty(B_{2\rho}(z_0))$ with $0\leq\eta\leq1$, $\eta=1$ on $B_\rho(z_0)$ and
$|\nabla\eta|\leq\frac2\rho$. The function $\eta^pu$ is Lipschitz with compact support in
$\Omega$, hence admissible in \eqref{eq:eq-f} by density, and
\begin{equation*}
\int\eta^p|\nabla u|^p\,dx
=-p\int\eta^{p-1}u\,|\nabla u|^{p-2}\nabla u\cdot\nabla\eta\,dx+\int f(u)\,\eta^pu\,dx .
\end{equation*}
Here $|f(u)\eta^pu|\leq C_E\eta^p|u|^p$ by \eqref{eq:f-growth}, while Young's inequality gives
\begin{equation*}
p\,\eta^{p-1}|\nabla u|^{p-1}\,|u|\,|\nabla\eta|
\ \leq\ \tfrac12\,\eta^p|\nabla u|^p+C(p)\,|u|^p|\nabla\eta|^p .
\end{equation*}
Since $\rho\leq1$,
\begin{equation}
\label{eq:caccioppoli}
\int_{B_\rho(z_0)}|\nabla u|^p\,dx\ \leq\ \frac C{\rho^p}\int_{B_{2\rho}(z_0)}|u|^p\,dx ,
\qquad C=C(p,C_E) .
\end{equation}
See e.g.\ Lindqvist~\cite[pp.~9--10]{Lindqvist}.

\emph{Step 4: putting together \eqref{eq:R-below} and \eqref{eq:caccioppoli} on $\mathcal R$.} Since $|v|\leq1$ and $|h|=\frac12|\nabla u|$, we have
$|\mathcal R|\leq\frac12|\nabla u|+|u|$, hence
$|\mathcal R|^p\leq2^{p-1}\big(2^{-p}|\nabla u|^p+|u|^p\big)\leq|\nabla u|^p+2|u|^p$, as $p<2$. By
\eqref{eq:caccioppoli}, and since $\rho\leq1$, for $0<\rho<\rho_0$
\begin{equation*}
\int_{B_\rho(z_0)}|\mathcal R|^p\,dx\ \leq\ \int_{B_\rho(z_0)}|\nabla u|^p\,dx+2\int_{B_\rho(z_0)}|u|^p\,dx
\ \leq\ \frac{C+2}{\rho^p}\int_{B_{2\rho}(z_0)}|u|^p\,dx ,
\end{equation*}
and by \eqref{eq:R-below}
\begin{equation*}
\int_{B_{2\rho}(z_0)}|u|^p\,dx\ \geq\ \frac{c''}{C+2}\,\rho^{nKp+p+2} ,
\end{equation*}
hence \eqref{eq:finite-order} follows with
$m=nKp+p+2=\frac{np}{p-1}+p+2$ and for some $c$ depending on $m$ and on the previous constants.  

\emph{Conclusion of the proof of (ii).} If $u\not\equiv0$ vanished
to infinite order at $z_0$, then, $m$ and $c$ being those just produced,
$c\,\rho^{m}\leq C_j\,\rho^{j}$ for small $\rho$ with $j>m$, which fails as $\rho\to0$, and this concludes the proof. 
\end{proof}

\section{Nonexistence on the half-plane: proof of Theorem B}
\label{sec:result}

\begin{lemma}[{\cite[Lemma 2.3]{MercuriWillem}}; {\cite[Lemma 2.1]{FarinaMercuriWillem}}]
\label{lem:normal}
Let $1<p<N$ and let $u\in\DR_0^{1,p}(H)$ be a weak solution of \eqref{eq:main}. Then $u\in C^{1,\alpha}_{{\textrm{loc}}}(\overline H)$ for some $\alpha\in(0,1)$, and
\begin{equation}
\label{eq:normal}
\partial_Nu=0\qquad\mbox{everywhere on }\partial H .
\end{equation}
\end{lemma}

We recall that the statement on regularity can be achieved as in \cite[p.~472]{MercuriWillem}, considering the odd reflection $\bar u$
of $u$ across $\partial H,$ a weak solution of \eqref{eq:main} on $\R^N$. It turns out that
$\bar u$ is locally bounded. The bound is obtained as in Trudinger's regularity theorem for the
Yamabe equation~\cite[Theorem 3]{Trudinger}. Tested against truncated powers of $\bar u$,
the equation gives $\bar u\in L^{qp^*}_{{\textrm{loc}}}(\R^N)$ for some $q>1$; hence
$-\Delta_p\bar u=a\,|\bar u|^{p-2}\bar u$ with $a=|\bar u|^{p^*-p}\in L^r_{{\textrm{loc}}}(\R^N)$, $r>N/p$, and
the solutions of such an equation are bounded by the truncation method of
Stampacchia~\cite[Lemme 4.1, p.~212; Th\'eor\`eme 4.1 and Remarque 4.2, pp.~213--215]{Stampacchia};
in this subcritical situation, $a\in L^r$ with $r>N/p$, the bound for quasilinear equations is the
$L^\infty$ estimate of Serrin~\cite{Serrin}, as indicated in Peral~\cite[p.~97]{Peral}.
For the $p$-Laplacian the two steps are performed in Appendix E of Peral's lecture notes~\cite[pp.~97--102]{Peral}. Hence, it holds that $\bar u\in C^{1,\alpha}_{{\textrm{loc}}}(\R^N)$ by DiBenedetto~\cite[p.~830]{DiBenedetto} and
Tolksdorf~\cite[Theorem 1]{Tolksdorf}. For the present range $1<p<2,$ we are not assuming any sign condition on the solutions.

\begin{proof}[Proof of Theorem~\ref{thm:main}]
Let $N=2$, $1<p<2$, and let $u\in\DR_0^{1,p}(H)$ solve \eqref{eq:main}. By Lemma~\ref{lem:normal},
$u$ and $\nabla u$ vanish on $\partial H$, and we extend $u$ by zero to $\R^2\setminus H$. We claim
that $u$ is a weak solution of $-\Delta_pu=|u|^{p^*-2}u$ in $\R^2$. Let $\varphi\in C_c^\infty(\R^2)$,
and let $\eta\in C^\infty(\R)$ with $\eta=0$ on $(-\infty,1]$, $\eta=1$ on $[2,\infty)$,
$\eta_k(x)=\eta(kx_2)$. Then $\varphi\eta_k\in C_c^\infty(H)$ is an admissible test function in
\eqref{eq:main}:
\begin{equation}
\label{eq:test-k}
\int_H\eta_k\,|\nabla u|^{p-2}\nabla u\cdot\nabla\varphi
+\int_H\varphi\,|\nabla u|^{p-2}\partial_2u\;k\eta'(kx_2)
=\int_H|u|^{p^*-2}u\,\varphi\,\eta_k .
\end{equation}
The second integral is in fact over a truncated shrinking strip
$S_k=\{0<x_2<2/k\}\cap\operatorname{supp}\varphi$, where $|\nabla u|^{p-1}\leq\varepsilon_k\to0$, since
$\nabla u$ is continuous up to $\partial H$ and vanishes there, while
$\int_{S_k}k|\eta'(kx_2)|\,dx\leq\|\eta'\|_\infty\,\mathcal H^{1}(\pi(\operatorname{supp}\varphi))$,
$\pi$ the orthogonal projection onto $\partial H$; so it tends to $0$. Letting $k\to\infty$
in \eqref{eq:test-k} by the dominated convergence theorem,
\begin{equation*}
\int_{\R^2}|\nabla u|^{p-2}\nabla u\cdot\nabla\varphi=\int_{\R^2}|u|^{p^*-2}u\,\varphi .
\end{equation*}
Since $u$ vanishes on the open half-plane $\R^2\setminus\overline H$, and $f(s)=|s|^{p^*-2}s$
satisfies \eqref{eq:f-growth} with $C_M=M^{p^*-p}$, Theorem~\ref{thm:ucp}(i) gives $u\equiv0$.
\end{proof}

\begin{rem}
\label{rem:local}
It is possible to relax the finite energy assumption as follows. Arguing as in
\cite[ p.~473]{MercuriWillem} and \cite[ p.~6]{FarinaMercuriWillem}, it is
easy to see that the local Pohozaev identity used in Lemma~\ref{lem:normal}, together with the
$C^{1,\alpha}_{{\textrm{loc}}}(\overline H)$ regularity, hold for every
$u\in W^{1,p}_{{\textrm{loc}}}(\overline H)$ solving \eqref{eq:main} with zero trace on
$\partial H$, and that $\partial_2u=0$ on $\partial H$ provided
\begin{equation*}
\liminf_{R\to\infty}\ \frac1R\int_{H\cap B_R}\big(|\nabla u|^p+|u|^{p^*}\big)\,dx=0 ,
\end{equation*}
since otherwise the integral is bounded below by $cR$ for $R$ large. For $p=2$ this corresponds to
\cite[Remark I.11]{EstebanLions}, where their Theorem I.1 is proved for
$u\in W^{2,q}(\Omega\cap B_R)$, $q<\infty$, for every $R>0$, provided
$\int_{\Omega\cap\partial B_{R_n}}(|F(u)|+|Du|^2)\,ds\to0$ along a sequence $R_n\to\infty$. Theorem~\ref{thm:main}
holds therefore for such local solutions. We point out that some condition at infinity is needed: for
a periodic solution $w$ of $(|w'|^{p-2}w')'+|w|^{p^*-2}w=0$ with $w(0)=0$ and $w'(0)\neq0$,
$u(x)=w(x_2)$ is a bounded solution of \eqref{eq:main} with $\partial_2u\equiv w'(0)$ on
$\partial H$.
\end{rem}

\section{Nonexistence in star-shaped domains: proof of Theorem C}
\label{sec:star}

\begin{proof}[Proof of Theorem~\ref{thm:star}]
We may assume $x_0=0$.

\emph{Step 1: regularity up to the boundary, and the Pohozaev identity.} The regularity holds up to the boundary,
$u\in C^{1,\alpha}(\overline \Omega)$ for some $\alpha\in(0,1)$, by Guedda and
V\'eron~\cite[Corollary 1.1]{GueddaVeron}, who provide in the same
paper~\cite[Theorem 1.1]{GueddaVeron} the Pohozaev-type identity
\begin{equation}
\label{eq:pohozaev}
\begin{split}
&\Big(\frac 2{p^*}+1-\frac 2p\Big)\int_\Omega|u(x)|^{p^*}dx
=\Big(1-\frac1p\Big)\int_{\partial \Omega}\big(x\cdot\nu(x)\big)\,\big|\partial_\nu u(x)\big|^p\,d\sigma(x).
\end{split}
\end{equation}
where the left hand side is zero by definition of $p^*$.

\emph{Step 2: $\nabla u$ vanishes on part of the boundary.} As $u=0$ on $\partial \Omega$ and
$u\in C^1(\overline \Omega)$, $\nabla u=(\partial_\nu u)\,\nu$ on $\partial \Omega$. By
\eqref{eq:star-shaped} the integrand on the right of \eqref{eq:pohozaev} is continuous and
nonnegative, and its integral vanishes; hence
\begin{equation}
\label{eq:Gamma}
\nabla u=0\qquad\mbox{on }\Gamma=\{x\in\partial \Omega:\ x\cdot\nu(x)>0\},
\end{equation}
a relatively open subset of $\partial \Omega$. It is nonempty: if $\overline{B_\rho(0)}$ is the
smallest closed disk centred at $0$ containing $\overline \Omega$ and
$y\in\partial \Omega\cap\partial B_\rho(0)$, then $\Omega\subset B_\rho(0)$ and the $C^1$ regularity
of $\partial \Omega$ give $\nu(y)=y/\rho$, so $y\cdot\nu(y)=\rho>0$
(Guedda and V\'eron~\cite[p.~886]{GueddaVeron}).

\emph{Step 3: extending $u$ as a solution across a piece of $\Gamma$ and unique continuation.} Fix $y\in\Gamma$. Since
$\partial \Omega$ is $C^2$, after a rigid motion there are a disk $B=B_r(y)$ and $\gamma\in C^2(\R)$
with
\begin{equation}
\label{eq:graph}
\Omega\cap B=\{x\in B:\ x_2>\gamma(x_1)\},\qquad
\partial \Omega\cap B=\{x\in B:\ x_2=\gamma(x_1)\}\subset\Gamma .
\end{equation}
Let $\widetilde\Omega=\Omega\cup B$, and extend $u$ by zero on $B\setminus\Omega$. By \eqref{eq:Gamma} and \eqref{eq:graph},
$u\in C^1(\widetilde\Omega)\cap L^\infty(\widetilde\Omega)$. We claim also that is a weak solution to
to the same equation in $\widetilde\Omega.$ In fact, for $\varphi\in C_c^\infty(B)$,
with $\eta$ as in the proof of Theorem~\ref{thm:main} and $\eta_k(x)=\eta(k(x_2-\gamma(x_1)))$, the
function $\varphi\eta_k$ is $C^1$ with compact support in $\Omega\cap B$, hence it is an admissible
test function for \eqref{eq:eq-Omega}:
\begin{equation*}
\int_{\Omega\cap B}\eta_k\,|\nabla u|^{p-2}\nabla u\cdot\nabla\varphi\,dx
+\int_{\Omega\cap B}\varphi\,|\nabla u|^{p-2}\nabla u\cdot\nabla\eta_k\,dx
=\int_{\Omega\cap B}|u|^{p^*-2}u\,\varphi\,\eta_k\,dx ,
\end{equation*}
and letting $k\to\infty$ as in the proof of Theorem~\ref{thm:main} we obtain
\begin{equation*}
\int_B|\nabla u|^{p-2}\nabla u\cdot\nabla\varphi\,dx=\int_B|u|^{p^*-2}u\,\varphi\,dx .
\end{equation*}
We conclude by Theorem~\ref{thm:ucp}(i) with $C_M=M^{p^*-p}$ that $u\equiv0$ on $\Omega$ too, and this concludes the proof.
\end{proof}

\section{Blow-up analysis in the plane: proof of Theorems D and E}
\label{sec:compact}

\begin{proof}[Proof of Theorem~\ref{thm:compact}]
Since the proof would be a repetition of that of \cite[Theorem 1.2, p.~471]{MercuriWillem}, we leave it out. Its
proof is in the spirit of Willem~\cite[Theorem 8.13, pp.~128--129]{Willem} for $p=2$. We complete the
proof with two remarks as follows.\newline
\emph{Boundedness of PS sequences.} This is proved for simplicity in \cite{MercuriWillem}, under an extra assumption. We can remove it, by the classical argument of Brezis and
Nirenberg~\cite[p.~463]{BrezisNirenberg}, noting that
\begin{equation}
\label{eq:PS-bound}
\Big(\frac1p-\frac1{p^*}\Big)\int_\Omega|u_n|^{p^*}dx
=J_\Omega(u_n)-\frac1p\langle J_\Omega'(u_n),u_n\rangle\ \leq\ c+1+\|u_n\| ,
\end{equation}
and that by H\"older's inequality with $\frac p2+\frac p{p^*}=1$,
\begin{equation}
\label{eq:PS-bound2}
\begin{split}
\|u_n\|^p&=p\,J_\Omega(u_n)-\int_\Omega a|u_n|^p\,dx+\frac p{p^*}\int_\Omega|u_n|^{p^*}dx\\
&\leq\ p(c+1)+\|a\|_{L^{2/p}(\Omega)}\|u_n\|_{L^{p^*}(\Omega)}^p+\frac p{p^*}\int_\Omega|u_n|^{p^*}dx .
\end{split}
\end{equation}
By \eqref{eq:PS-bound} it follows then that the right hand side of \eqref{eq:PS-bound2} is at most
$C\big(1+(1+\|u_n\|)^{p/p^*}+\|u_n\|\big)$, and $p/p^*<1<p,$ hence $(u_n)_{n\in\mathbb N}$ is bounded in
$W_0^{1,p}(\Omega)$. 

\emph{Half-plane profiles vanish.} A half-plane is produced by the blow-up at a concentration
point when the points $y_n\in\Omega$ and the dilations $\varepsilon_n$ satisfy, along a
subsequence,
\begin{equation}
\label{eq:halfplane-blowup}
\varepsilon_n\to0,\qquad
\frac{\operatorname{dist}(y_n,\partial\Omega)}{\varepsilon_n}\ \mbox{bounded} :
\end{equation}
then $y_n$ converges to a
point of $\partial\Omega$ and the rescaled functions
$\varepsilon_n^{\frac{2-p}p}u_n(\varepsilon_n\,\cdot+y_n)$ converge to a solution of the limiting
equation on a half-plane, vanishing on its boundary (\cite[pp.~480--481]{MercuriWillem}). In the opposite case
$\operatorname{dist}(y_n,\partial\Omega)/\varepsilon_n\to\infty$ the limit solves the equation on
all of $\R^2$ (\cite[p.~477]{MercuriWillem}), and these are the profiles mentioned in
\eqref{eq:decomposition}. Let now $v\in\DR_0^{1,p}(H)$ be a weak solution of
\eqref{eq:main}, the half-plane of the limit being brought to $H$ if necessary by a translation and a rotation.
By Theorem~\ref{thm:main}, $v\equiv0$. This contradicts the nontriviality of the rescaled limit,
ensured by the choice of $y_n$ and $\varepsilon_n$ (\cite[step 3 of the proof of Theorem 1.2,
p.~479]{MercuriWillem}); hence \eqref{eq:halfplane-blowup} never occurs, and it holds that
$\operatorname{dist}(y_n,\partial\Omega)/\varepsilon_n\to\infty.$ 
\end{proof}

\begin{proof}[Proof of Theorem~\ref{thm:chernysh}]
The profiles on the half-space of Chernysh~\cite[Theorem 2, p.~3]{Chernysh} are nontrivial finite energy
solutions, for
$N=2$ to $-\Delta_pu=|u|^{p^*-2}u$ in $H.$ 
By
Theorem~\ref{thm:main} every such solution vanishes, so that we may conclude that no such profile occurs.
\end{proof}

\section*{Acknowledgements}
Carlo Mercuri is a member of the group GNAMPA of Istituto Nazionale di Alta Matematica (INdAM).

\section*{Statements and Declarations}

\noindent\textbf{Funding.} The author did not receive support from any organization for the submitted work.

\noindent\textbf{Competing interests.} The author has no relevant financial or non-financial interests to disclose.

\end{document}